\documentclass[a4paper,fleqn]{cas-sc}

\usepackage[numbers]{natbib}

\def\tsc#1{\csdef{#1}{\textsc{\lowercase{#1}}\xspace}}
\tsc{WGM}
\tsc{QE}
\tsc{EP}
\tsc{PMS}
\tsc{BEC}
\tsc{DE}

\usepackage{amsthm}

\newtheorem{theorem}{Theorem}[section]

\newtheorem{lemma}[theorem]{Lemma}
\newtheorem{proposition}{Proposition}[section]

\theoremstyle{definition}
\newtheorem{definition}[theorem]{Definition}
\newtheorem{property}[theorem]{Property}
\newtheorem{remark}{Remark}[section]

\newtheorem{example}{Example}[section]

\newcommand{\grad}{\nabla}
\newcommand{\Diag}{\mbox{Diag}}
\newcommand{\D}{\mathbf{D}}
\renewcommand{\H}{\mathbf{H}}
\newcommand{\norm}[1]{\left\lVert#1\right\rVert}
\newcommand{\abs}[1]{\left|#1\right|}
\renewcommand{\epsilon}{\varepsilon}
\newcommand{\R}{\mathbb{R}}
\newcommand{\Epsilon}{\mathcal E}

\usepackage{subcaption}
\usepackage{multirow}
\usepackage{mathrsfs}
\usepackage{amsfonts,amsmath,bm}

\usepackage[dvipsnames]{xcolor}
\newcommand{\cblue}[1]{{\color{black}#1}}
\newcommand{\cred}[1]{{\color{black}#1}}
\newcommand{\cgreen}[1]{{\color{black}#1}}
\newcommand{\cpurple}[1]{{\color{black}#1}}

\begin{document}
\sloppy
\let\WriteBookmarks\relax
\def\floatpagepagefraction{1}
\def\textpagefraction{.001}
\shorttitle{Energy stable and accurate nondiagonal norm interpolation operators}
\shortauthors{T.A. Dao et~al.}

\title [mode = title]{Energy stable and accurate coupling of finite element methods and finite difference methods}                      
\tnotemark[1]

\tnotetext[1]{This research was funded by Uppsala University.}

\author[1]{Tuan Anh Dao}[orcid=0000-0002-5591-0373]
\cormark[1]
\ead{tuananh.dao@it.uu.se}

\address[1]{Department of Information Technology, Uppsala University, Sweden}

\author[1]{Ken Mattsson}[orcid=0000-0003-0843-3321]
\ead{ken.mattsson@it.uu.se}

\author[1]{Murtazo Nazarov}[orcid=0000-0003-4962-9048]
\ead{murtazo.nazarov@it.uu.se}

\cortext[cor1]{Corresponding author}

\begin{abstract}
We introduce a hybrid method to couple continuous Galerkin finite element methods and high-order finite difference methods in a nonconforming multiblock fashion. The aim is to optimize computational efficiency when complex geometries \cpurple{are} present. The proposed coupling technique requires minimal changes in the existing schemes while maintaining strict stability, accuracy, and energy conservation. Results are demonstrated on linear and nonlinear scalar conservation laws in two spatial dimensions.
\end{abstract}




\maketitle

\section{Introduction}
\vspace{0.3cm}
Finite element (FE) methods and finite difference (FD) methods are among the most studied and employed numerical methods for partial differential equations (PDEs), which serve numerous industrial applications and other research areas, e.g., economics, engineering, chemistry, physics. FD methods are often seen as \cpurple{simple and efficient schemes} for high-order approximations, see e.g., \citet{Mattsson2004,Shu2003}, despite being restricted to structured grids. FE methods are favored in various computational \cpurple{areas} due to their theoretical reliability, generality, and the ability to handle complex \cpurple{domain} shapes. However, the obtained matrices can be storage demanding; and the method itself is often considered to be more computationally expensive. For instance, while solving time-dependent nonlinear diffusion dominated problems, one has to reconstruct FE matrices in each time step, which is time-consuming. The greatest strengths of both methods can be combined by developing a hybrid scheme that allows different methods and refinement levels to be involved simultaneously. Hence, computational efficiency can be enhanced by utilizing the most suitable treatment for each portion of the domain. The major complication is known to consist in designing a method-to-method interface that results in an accurate and stable approximation. \cpurple{A favorable method should also accommodate the demands for simplicity, freedom to choose domain discretization, and minimal modification in the existing schemes to preserve their advantageous native properties.}

\cpurple{A well-designed numerical framework for energy stable approximations of time-dependent problems is the combination of the summation-by-parts (SBP) operators, see \citet{Kreiss1974}, and the simultaneous-approximation-term (SAT) technique, see \citet{Carpenter1994}. The SBP operators are used to approximate the governing equations. By design, they mimic the continuous integrations-by-parts properties. The SAT technique then imposes the boundary conditions in a provably stable manner.} We refer the reader for a detailed discussion on the theory and applications of SBP-SAT in the following papers and references therein, \cpurple{\cite{Carpenter2010,Fernandez2014,Kozdon2016,Lundgren2020,Lundquist2018,Mattsson2004,Nordstrom2001,Zemui2005}}. A crucial benefit of numerical schemes in this form is that a proper formulation strictly prevents nonphysical energy growth in the discrete approximation, a property often referred to as ``\textit{strict stability}''. The vast majority of previous papers on the SBP-SAT framework have been committed to FD schemes due to high-order operators' apparent advantages. Nontrivial geometries are then usually handled with \cpurple{curvilinear grids, see \citet{Nordstrom2001}, \cblue{\citet{Shadpey2020}}, \cblue{curved elements, see \citet{Crean2018}}, and multiblock coupling, see \citet{Carpenter2010}}. To carry out further practical geometries, successful attempts have been made to couple FD methods with unstructured finite volume (FV) methods, see \citet{Nordstrom2006,Nordstrom2009} by modifying the FV scheme at interfaces. The coupling of mixed order schemes on nonconforming grids was, for the first time, \cpurple{proposed} in an energy stable manner by \citet{Mattsson2010} introducing \textit{SBP-preserving interpolation operators}. Fundamental properties of SBP schemes are extended to block-to-block coupling: strict stability, accuracy, and conservation. This technique paves the way for many possibilities to couple different methods which can be written in SBP form. \cpurple{So far, the coupling of  FD--FD by \citet{Mattsson2010}, FD--discontinuous Galerkin (dG) method by \citet{Kozdon2016}, and FD--FV by \citet{Lundquist2018} are available in the literature.}

In practical computation, properties of the concerning problem, e.g., nature of the solution, desired precision, geometries, often vary over the domain. An impediment with discretizations by a single conventional method, especially when using uniform precision, is that sometimes the preferable scheme or the required resolution for a minor region may bring up extreme computational necessity over the whole domain, e.g., when complex geometries present. Despite being of great interest, attempts to combine FD methods and FE methods have rarely been successful. Independent techniques have been proposed \cpurple{for} several applications, e.g., seismic wave propagation, see \cite{Galis2008,Ma2004,Gao2019}, electrical modeling, see \cite{Vachiratienchai2010}. For example, \cite{Galis2008,Ma2004} require a transition region between the two methods; the interface solution is updated similarly to imposing Dirichlet boundary conditions with a suitably averaging solution. In \cite{Vachiratienchai2010}, at regions of interest, each rectangular FD block is \cpurple{split} into two FE triangular elements resulting in locally modified schemes; accuracy is observed in cases of sufficiently gradual topographic gradients. Most relevantly, \citet{Gao2019} derive an energy stable coupling of SBP FD schemes and FE schemes with quadrilateral elements under shared interface vertices/nodal points. Inline with other hybrid SBP methods, although \cpurple{the continuous Galerkin FE also has SBP properties, see \citet{Zemui2005}}, \cpurple{one difficulty} in deriving a hybrid method lies within its nondiagonal mass matrix (can be understood as an SBP norm matrix). Note that the most related references, such as \citet{Lundquist2018,Gao2019}, necessitate diagonal norm matrices.

In this paper, \cred{our main focus is to develop} a hybrid method to couple FD and FE methods in a nonconforming multiblock fashion. The proposed technique results in a provably stable, accurate, and energy-conserved unified SBP method. Interface continuity is weakly imposed by the SAT technique using \textit{nondiagonal norm SBP-preserving interpolation operators}. The SAT treatment on the FD side is the same as in our FD--FD coupling \cite{Mattsson2010}. \cpurple{The SAT treatment for FE that we propose in this current paper} is formulated in a general expression in the sense that it does not depend on \cpurple{the coupled scheme or the coupled mesh/grid}. \cred{Furthermore}, we show that even when considering nonlinear conservation laws, nondiagonal norm matrices, namely the continuous FE mass matrices, could be applicable without further \cred{techniques involved}, e.g., mass lumping \cite{Zemui2005}, to the classical Galerkin formulation. \cred{We include two techniques to construct the required interpolation operators}. In the latter one, we introduce a \cred{way} of acquiring the target interpolation operators through existing diagonal norm techniques, for instance, \cite{Almquist2018,Kozdon2016,Lundquist2018,Lundquist2020,Mattsson2010}.

This paper is organized as follows. In Section \ref{sec:sbp_operators}, the numerical schemes of both FE and FD are \cpurple{presented} in the context of SBP framework. \cpurple{Section \ref{sec:coupling} describes the proposed coupling technique}. We show numerical verification and computational results in Section \ref{sec:numerical_results} on an advection-diffusion equation and a nonlinear Burgers' equation. Section \ref{sec:conclusion} is the conclusion. Two approaches to constructing the required interpolation operators are shown in Appendix \ref{section:construct_interpolation_operators}.

\section{Summation-by-parts operators}
\label{sec:sbp_operators}
This section formulates summation-by-parts schemes by FD and FE approximations. Boundary treatment is briefly \cpurple{mentioned}.
\subsection{The continuous problem to a viscous scalar conservation law}
\label{sec:advection_diffusion_continuous}
Let $\Omega$ be a bounded domain in $\R^2$. \cpurple{Consider the following scalar conservation laws}
\begin{equation}
\label{eq:advection_diffusion_fem}
\begin{aligned}
\frac{\partial u}{\partial t} + \grad \cdot \bm{f}(u) & =  \grad \cdot (\epsilon \grad u), && (x,y) \in \Omega, t\in (0,+\infty),\\
u & = u_0(x,y), && (x,y) \in \Omega, t = 0,
\end{aligned}
\end{equation}
where the flux vector $\bm{f}(u) = \left(f_1(u),f_2(u)\right)^T$ consists of two linear or nonlinear mappings $f_1,f_2 \in C^1(\Omega; \mathbb R)$; $\epsilon = \epsilon(x,y)$ is a positive-valued function on $\Omega$; and $\mathbf{n}$ is the normal vector pointing outward at the boundary. The notation ``$\grad$'' refers to the gradient operator $\grad = \left(\frac{\partial}{\partial x},\frac{\partial}{\partial y}\right)^T$. In analysis of FD schemes, it is more convenient to write \eqref{eq:advection_diffusion_fem} as
\begin{equation}
\label{eq:advection_diffusion_fdm}
u_t + (f_1)_{x}u_x + (f_2)_{y}u_y = (\epsilon u_{x})_x+(\epsilon u_{y})_y, \quad (x,y) \in \Omega,
\end{equation}
where the subscript notations $x,y,t$ indicate the partial derivatives with respect to the corresponding variables. \cblue{Let $\bm f' = (f_1', f_2')^T$ be the first derivative of $\bm f$ with respect to the conserved variable $u$}. We assume that there exists a skew-symmetric splitting with the weight functions $\bm\alpha,\bm{\tilde\alpha}$,
\begin{equation}
\label{eq:skew_symmetric_assumption}
\grad \cdot \bm{f} = (\bm\alpha\grad) \cdot \bm{f} + \bm f'(u) \cdot \big(\bm{\tilde\alpha}\grad\big) u, \quad\bm\alpha = (\alpha_1(x,y),\alpha_2(x,y))^T,\bm{\tilde\alpha}=(1-\alpha_1(x,y),1-\alpha_2(x,y))^T,
\end{equation}
such that by using the alternative flux form, the total energy over time of a solution to \eqref{eq:advection_diffusion_fem} only conserves or decays, on the boundary or due to viscous dissipation, yet to be quantified. It is \cpurple{worth noting} that the above assumption is not too severe, in which the existence of such a splitting is shown for a broad class of fluxes \cpurple{by} \citet{Tadmor1984} by symmetrizing variables. \cpurple{This splitting technique is necessary for obtaining nonlinear energy estimates to \eqref{eq:advection_diffusion_fem}}. In \eqref{eq:skew_symmetric_assumption}, the following notations are used $\bm\alpha\grad=\left(\alpha_1\partial/\partial x,\alpha_2\partial/\partial x\right)^T,\bm{\tilde\alpha}\grad=\left(\tilde\alpha_1\partial/\partial x,\tilde\alpha_2\partial/\partial x\right)^T$. Similar notations used below are $(\bm\alpha\bm n)=(\alpha_1 n_1,\alpha_2 n_2)^T,(\bm{\tilde\alpha}\bm n)=(\tilde\alpha_1n_1,\tilde\alpha_2n_2)^T,\bm n=(n_1,n_2)^T$. One option of specifying the boundary condition is
\begin{equation}
\label{eq:advection_diffusion_fem_bc}
\begin{aligned}
\frac{1}{2}\big((\bm\alpha\mathbf{n})\cdot \bm{f}-\text{sign}(u)\abs{(\bm\alpha\mathbf{n})\cdot \bm{f}}\big)+\frac{1}{2}\big((\bm{\tilde\alpha}\mathbf{n})\cdot (u\bm{f}')-\text{sign}(u)\abs{(\bm{\tilde\alpha}\mathbf{n})\cdot (u\bm{f}')}\big)-\mathbf{n}\cdot\epsilon\grad u & = g_b(x,y,t),\\ & (x,y) \in \partial \Omega, t > 0,
\end{aligned}
\end{equation}
where $\partial \Omega$ denotes the domain boundary. \cgreen{Stability} of the problem \eqref{eq:advection_diffusion_fem} subject to the boundary condition \eqref{eq:advection_diffusion_fem_bc} is analyzed \cpurple{below}.
\begin{definition}
\label{def:continuous_inner_product_norm}
Given two real-valued vector functions $\bm w_1$, $\bm w_2$, and a real-valued function $b(x,y)>0$, the following definitions of a continuous \textit{inner product} and its corresponding \textit{norm} are used
\[
(\bm w_1,\bm w_2)_b = \int_\Omega \bm w_1\cdot \bm w_2 b(x,y) \,dxdy,\quad \norm{\bm w_1}^2_b = (\bm w_1,\bm w_1)_b,\quad(\bm w_1,\bm w_2)_{\partial\Omega,b} = \int_{\partial\Omega} \bm w_1\cdot \bm w_2 b(x,y) \,ds.
\]
\end{definition}
\noindent When $b(x,y)\equiv 1$, the subscript notation $b$ is skipped. Inner products and a norm of two scalar functions are defined analogously by considering $\bm w_1,\bm w_2$ as scalar functions. One can investigate necessary conditions for well-posedness of \eqref{eq:advection_diffusion_fem} with the boundary condition \eqref{eq:advection_diffusion_fem_bc} using the \textit{energy method}, see e.g., \citet{Kreiss1974,Fernandez2014}. For continuous problems that can be written in the form $u_t=F(t,u)$, a basic idea of the energy method is to examine energy growth utilizing the relation $\frac{d}{dt}\norm{u}^2=(u,u_t)+(u_t,u)$. A well-posed problem implies that $\frac{d}{dt}\norm{u}^2$ is nonpositive given zero boundary data. For two smooth scalar functions $w_1,w_2\in C^1(\Omega)$ and a smooth vector function $\bm{w}_3$, continuous integration rules read
\begin{equation}
\label{eq:continuous_intergration_rules}
\begin{aligned}
\left(w_1,\grad \cdot \bm{w}_3\right) & = \left(w_1,\mathbf{n}\cdot \bm{w}_3\right)_{\partial\Omega} - \left(\grad w_1,\bm{w}_3\right),\\
\left(w_1,\grad\cdot\grad w_2\right) & = \left(w_1,\mathbf{n}\cdot\grad w_2\right)_{\partial\Omega} - \left(\grad w_1, \grad w_2\right).
\end{aligned}
\end{equation}
Applying \eqref{eq:continuous_intergration_rules} to \eqref{eq:advection_diffusion_fem} gives
\begin{align*}
(u,u_t) = &\; \left(\bm\alpha\grad u,\bm{f}\right)-\left(u,\bm\alpha\mathbf{n}\cdot \bm{f}\right)_{\partial\Omega}+(\tilde{\bm\alpha}\grad\cdot(u \bm f'),u)-(\tilde{\bm\alpha}\mathbf{n}\cdot(u\bm{f}'),u)_{\partial\Omega}\\
&+(u,\mathbf{n}\cdot\epsilon\grad u)_{\partial\Omega}-(\epsilon\grad u,\grad u),\\
(u_t,u) = &\;-\left(\bm\alpha\grad \cdot \bm f(u),u\right)-\left(\tilde{\bm\alpha}\bm f'(u) \cdot \grad u, u \right)\\
&+(\mathbf{n}\cdot \epsilon\grad u,u)_{\partial\Omega}-(\epsilon\grad u,\grad u).
\end{align*}
Adding up the two equalities above yields
\begin{align}\label{eq:continous_energy_estimate}
\frac{d}{dt} \norm{u}^2 + 2\norm{\grad u}_\epsilon^2 = -(u,\bm\alpha\mathbf{n}\cdot\bm{f})_{\partial\Omega}-(\tilde{\bm\alpha}\bm n \cdot (u\bm f'),u)_{\partial \Omega}+2(u,\mathbf{n}\cdot\grad u)_{\partial \Omega,\epsilon},
\end{align}
where all the interior terms are cancelled by the following skew-symmetry assumption
\begin{equation}
\label{eq:skew_symmetric_precise}
-\alpha_1[(f_1)_xu-f_1u_x]-\alpha_2[(f_2)_yu-f_2u_y]+(1-\alpha_1)u^2(f_1')_x+(1-\alpha_2)u^2(f_2')_y = 0.
\end{equation}
\cred{If the splitting is defined pointwise}, i.e., $\alpha_1, \alpha_2$ are constant weights, a sufficient spatial-derivative-free argument to \eqref{eq:skew_symmetric_precise} reads
\begin{equation}\label{eq:skew_symmetric_sufficient}
\begin{cases}
\alpha_1f_1(u)-(1-\alpha_1)uf_1'(u) = 0,\\
\alpha_2f_2(u)-(1-\alpha_2)uf_2'(u) = 0.
\end{cases}
\end{equation}
Given a specific flux, one can either solve \eqref{eq:skew_symmetric_precise} or \eqref{eq:skew_symmetric_sufficient} to determine the value of $\bm\alpha$. It becomes apparent that the boundary condition \eqref{eq:advection_diffusion_fem_bc} imposes nonlinear stability to \eqref{eq:advection_diffusion_fem} since bounded energy is followed. Indeed, inserting \eqref{eq:advection_diffusion_fem_bc} \cred{into \eqref{eq:continous_energy_estimate}} with zero boundary data $g_b \equiv 0$ gives
\begin{align}
\label{eq:advection-diffusion-continuous-energy}
\frac{d}{dt} \norm{u}^2 + 2\norm{\grad u}_\epsilon^2 = -(u,\text{sign}(u)\abs{\bm\alpha
\mathbf{n}\cdot \bm{f}})_{\partial \Omega}-(\text{sign}(u)\abs{\tilde{\bm\alpha}\bm n \cdot (u\bm f')},u)_{\partial \Omega} \leq 0.
\end{align}

\subsection{Properties of the discrete operators}
\cgreen{Considering a general numerical method}, we denote \cblue{$\mathrm N$} the total number of degrees of freedom of the discretization, and \cblue{$\mathrm N_{\partial\Omega}$} the number of degrees of freedom distributed on the boundary. Let $\mathcal{H}$ be an integration operator over the whole domain $\mathcal{H}\bm u \approx \int_{\Omega} u dxdy, \bm u \in \R^{\cblue{\mathrm N}}, \bm u\approx u(x,y)$; $\mathcal{B}$ be an integration operator over the boundary $\mathcal{B}\bm u|_{\partial\Omega} \approx \int_{\partial\Omega} u dxdy$; and $\mathcal{L}$ be a boundary selection operator $\bm u|_{\partial\Omega}=\mathcal{L}\bm u$. Precise definitions of the three matrix operators $\mathcal{H},\mathcal{B},\mathcal{L}$ are often required by conventional summation-by-parts schemes, see e.g., \citet{Kreiss1974,Lundquist2018}. In the current study, $\mathcal{H},\mathcal{B},\mathcal{L}$ are assumed to have the following properties: $(i)$ $\mathcal{H},\mathcal{B}$ are symmetric, positive definite; $\mathcal{H}$ is of size \cblue{$\mathrm N\times\mathrm N$}; and $\mathcal{B}$ is of size \cblue{$\mathrm N_{\partial\Omega}\times\mathrm N_{\partial\Omega}$}; $(ii)$ $\mathcal{L}$ comprises a nonsquare binary matrix of which the only ``1'' element in each row one-to-one corresponds to a boundary index.

\begin{definition}
\label{def:discrete_inner_product_norm}
Given two vectors $\bm{u},\bm{v} \in \R^{\cblue{\mathrm N}}$, the following definitions of a \textit{discrete $\mathcal H$-inner product} and its corresponding \textit{discrete $\mathcal H$-norm} are used,
\[
(\bm{u},\bm{v})_{\mathcal{H}} = \bm{u}^T\mathcal{H}\bm{v};\quad (\mathcal{L}\bm{u},\mathcal{L}\bm{v})_{\mathcal{B}} = (\mathcal{L}\bm{u})^T\mathcal{B}(\mathcal{L}\bm{v});\quad \norm{\bm{u}}^2_{\mathcal{H}} = (\bm u,\bm u)_{\mathcal{H}} = \bm u^T\mathcal{H}\bm u.
\]
\end{definition}

\cgreen{The matrix $\mathcal{H}$ is often referred to as the norm matrix}. Advantages of an SBP scheme come from the fact that its discrete operators, by design, mimic the continuous integration-by-parts. The following definitions introduce discrete operators of which properties are analogous to the continuous integration rules \eqref{eq:continuous_intergration_rules}, see \citet{Carpenter1994,Nordstrom2001}.

\begin{definition}
\label{def:sbp_properties_first}
A matrix operator $\mathcal{H}^{-1}\mathcal{Q}$ approximating $\bm{a}\cdot\grad,\bm{a}\in\R^2$ is said to be an \textit{SBP advection operator} if
\begin{equation}
\label{eq:sbp_properties_first}
\mathcal{Q}+\mathcal{Q}^T = \mathcal{L}^T \mathcal{B} \;\mbox{Diag}(\bm a\cdot\mathbf{n})\; \mathcal{L},
\end{equation}
where $\Diag(\bm{\omega})$ denotes the square diagonal matrix of which diagonal is the discrete vector $\bm{\omega}$.
\end{definition}
\begin{definition}
\label{def:sbp_properties_second}
A matrix operator $\mathcal{H}^{-1}\mathcal{R}^{(\epsilon)}$ approximating $\grad\cdot\epsilon\grad$ is said to be an \textit{SBP diffusion operator} if
\begin{equation}
\label{eq:sbp_properties_second}
\mathcal{R}^{(\epsilon)} = \mathcal{L}^T\mathcal{B} \mathcal S-\mathcal{A},
\end{equation}
where $\mathcal{A}=\mathcal{A}^T \geq 0$ and $\mathcal S$ is a consistent approximation of $\mathbf{n}\cdot\epsilon\grad$ at the boundary.
\end{definition}
The discrete energy method can be analogously applied to SBP schemes utilizing the properties \eqref{eq:sbp_properties_first}, \eqref{eq:sbp_properties_second}. A necessary condition for strict stability is that $\frac{d}{dt}\norm{\bm u}_{\mathcal{H}}^2$ being nonpositive given zero boundary data.

%
We introduce FD notations in the following subsection.

\subsection{The FD scheme}
\label{sec:sbp_form_fdm}
Consider a one-dimensional domain $[x_l,x_r]$ discretized by a grid of $n$ equally spaced points $\{x_l\equiv x_1 < x_2<\dots<x_n \equiv x_r\}$. We first look at one-dimensional SBP operators for FD. The following two definitions (first presented in \cite{Mattsson2012, Mattsson2004}) are central.
\begin{definition}
A FD operator $\bm D_1 = \bm  H^{-1} \bm Q$ approximating $\frac{\partial}{\partial x}$ using $q$th-order interior stencils is called a \textit{$q$th-order accurate first derivative SBP operator} if $\bm H=\bm  H^T>0$ and $\bm Q+\bm Q^T=\Diag(-1,0,\dots,0,1)$.
\end{definition}

\cgreen{$\bm H$ is the one-dimensional FD norm matrix}. Let $\bm e_1 = (1,0,\dots,0)^T,\bm e_n = (0,\dots,0,1)^T$ be column vectors of length $n$, and $\bm d_1,\bm d_n$ be row vectors approximating first derivatives at $x_1,x_n$.

\begin{definition}
A FD operator $\bm D_2^{(\epsilon)} = \bm H^{-1}(-\bm A-\epsilon_1\bm e_1 \bm d_1+\epsilon_n\bm e_n \bm d_n)$ approximating $\epsilon(x)\frac{\partial^2}{\partial x^2}$ using $q$th-order interior stencils is called a \textit{$q$th-order accurate second derivative SBP operator} if $\bm H= \bm H^T>0, \bm A= \bm A^T\geq0$.
\end{definition}

\begin{figure}
    \centering
    \includegraphics[scale=1]{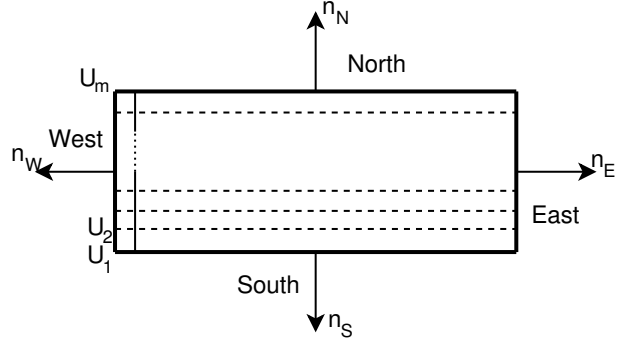}
    \caption{A rectangle domain discretized by an equidistant grid.}
    \label{fig:fd_domain}
\end{figure}

SBP FD schemes for multiple spatial dimensions are naturally extended from the one-dimensional operators. Consider a bounded rectangle domain $\Omega = [x_l,x_r]\times[y_l,y_r] \subset \R^2$ illustrated in Figure \ref{fig:fd_domain}. A FD solution is approximated on $n\times m$ equidistant grid points $\{(x_i,y_j), i =1,2,\dots,n, j = 1,2,\dots,m\}$, where
\begin{align*}
x_i&=x_l+(i-1)h_x, \quad h_x = \frac{x_r-x_l}{n-1},\\
y_j&=y_l+(j-1)h_y, \quad h_y = \frac{y_r-y_l}{m-1}.
\end{align*}

A discrete solution vector $\bm{u}$ approximating a continuous function $u:\Omega\longrightarrow\R$ is then presented by
\[
\bm{u} = \begin{bmatrix} u_{1,1},u_{1,2},\dots,u_{1,m}, &  u_{2,1},u_{2,2},\dots,u_{2,m}, & \dots &  u_{n,1},u_{n,2},\dots,u_{n,m} \end{bmatrix}^T
\]
where $u_{i,j}$ approximates $u(x_i,y_j)$. It's often convenient to interprete $\bm{u}$ as a ``vector of vectors''. The total number of degrees of freedom is $\cblue{\mathrm N}=n \times m$.\\

Discrete operators for the two-dimensional approximation are listed below. Definition of the Kronecker product ``$\otimes$'' can be found in e.g., \citet{Mattsson2010}. Denote the identity matrix by $\bm I$, and the one of size $k\times k$ by $\bm I_k$. The following FD operators will be frequently used in the \cblue{analysis},
\begin{equation}
\label{eq:fdm_2d_operators}
\begin{aligned}
\H_x & = \bm H_x \otimes\bm  I_m, & \H_y & =\bm  I_n \otimes\bm  H_y,\\
\D_x & = \bm D_1 \otimes\bm  I_m, & \D_y & = \bm I_n \otimes \bm D_1,\\
\D_{2x} & = \bm D_2^{(b)} \otimes \bm I_m, & \D_{2y} & =\bm  I_n \otimes\bm  D_2^{(b)},
\end{aligned}
\end{equation}
where $\bm H_x,\bm H_y$ are the one-dimensional norm matrices in $x-$ and $y-$ directions. For ease of reading, it is worth noting that the first components in all of the above Kronecker products are of size $n\times n$, and the second components are of size $m\times m$. Note that the matrix $\bm D_2^{(b)}$ in $\D_{2x},\D_{2y}$ varies for different rows and columns, namely,
\cblue{$\D_{2y}^{(b)} = \Diag(\bm D_2^{(\bm b^1)}, \dots, \bm D_2^{(\bm b^n)})$}, where $\bm b^i,i=1,2,\dots,n$, being a vector of length $m$, corresponds to the value of $b(x,y)$ at $(x_i,y_j),j=1,2,\dots,m$. The corresponding two-dimensional norm matrix is $\H:=\H_x\H_y = \bm H_x\otimes\bm  H_y$.

A consistent FD semi-discretization of \eqref{eq:advection_diffusion_fdm} using diagonal norm SBP operators is given by
\begin{equation}
\label{eq:advection_diffusion_fdm_approx}
\begin{aligned}
\bm{u}_t = & \; -\alpha_1 \D_x \bm f_1(\bm u) -(1-\alpha_1)\bm F_1' \D_x \bm{u} - \alpha_2 \D_y \bm f_2 (\bm{u}) -(1-\alpha_2)\bm F_2' \D_y \bm u\\
&+(\D_{2x}^{(\epsilon)}+ \D_{2y}^{(\epsilon)})\bm u + \mbox{SAT},
\end{aligned}
\end{equation}
where \cgreen{$\bm f_1, \bm f_1',\bm f_2, \bm f_2'$ are, respectively, the vectors containing values of $f_1,f_1',f_2,f_2'$ evaluated at nodal points with respect to the discrete variable $\bm u$}. The diagonal matrices $\bm F_1'$ and $\bm F_2'$ are defined as $\bm F_1' = \Diag (\bm f_1'(\bm u))$, $\bm F_2' = \Diag (\bm f_2'(\bm u))$.

The following proposition can be straightforwardly derived from how the above operators are defined, but useful in connection with FE schemes in the next section.
\begin{proposition}
\cred{Let $\grad \cdot \bm f$ be a linear advection flux, i.e. $\bm f'= \bm a = (a_1,a_2)^T\in\R^2$. }The advection operator $a_1\D_x+a_2\D_y, \bm a = (a_1,a_2) \in \R^2$, and the diffusion operator $\D_{2x}^{(\epsilon)} + \D_{2y}^{(\epsilon)}$ in \eqref{eq:advection_diffusion_fdm_approx} satisfy the summation-by-parts rules \eqref{eq:sbp_properties_first} and \eqref{eq:sbp_properties_second}, with the norm matrix $\H$.
\end{proposition}
Since this property is intended by definition, no proof is given.\\

The normal vectors at the North, East, South, West boundaries are $\mathbf{n}_N=(0,1)^T, \mathbf{n}_E=(1,0)^T,\mathbf{n}_S=(0,-1)^T$ and $\mathbf{n}_W=(-1,0)^T$, respectively. The boundary condition \eqref{eq:advection_diffusion_fem_bc} can be written more precisely at each boundary,
\[
\begin{cases}
\displaystyle\mbox{North}: \frac{1}{2} \left(\alpha_2f_2(u)-\text{sign}(u)\abs{\alpha_2f_2(u)}\right)+\frac{1}{2} \left(\tilde\alpha_2uf_2'(u)-\text{sign}(u)\abs{\tilde\alpha_2 uf_2' (u)}\right) -\epsilon u_y = g_{bN}(t),\\
\displaystyle\mbox{East}: \frac{1}{2} \left(\alpha_1f_1(u)-\text{sign}(u)\abs{\alpha_1f_1(u)}\right)+\frac{1}{2} \left(\tilde\alpha_1uf_1'(u)-\text{sign}(u)\abs{\tilde\alpha_1 uf_1'(u)}\right) -\epsilon u_x = g_{bE}(t),\\
\displaystyle\mbox{South}: \frac{1}{2} \left(-\alpha_2f_2(u)-\text{sign}(u)\abs{\alpha_2f_2(u)}\right)+\frac{1}{2} \left(-\tilde\alpha_2uf_2'(u)-\text{sign}(u)\abs{\tilde\alpha_2uf_2'(u)}\right) + \epsilon u_y = g_{bS}(t),\\
\displaystyle\mbox{West}: \frac{1}{2} \left(-\alpha_1f_1(u)-\text{sign}(u)\abs{\alpha_1f_1(u)}\right)+\frac{1}{2} \left(-\tilde\alpha_1uf_1'(u)-\text{sign}(u)\abs{\tilde\alpha_1uf_1'(u)}\right)+\epsilon u_x = g_{bW}(t),
\end{cases}
\]
where $g_{bN}, g_{bE}, g_{bS}, g_{bW}$ are parts of $g_b(x,y,t)$ on the North, East, South, West boundaries, respectively. We show that the following penalty terms weakly forcing the boundary condition \eqref{eq:advection_diffusion_fem_bc} yield stability to \eqref{eq:advection_diffusion_fdm_approx},
\begin{equation}
\begin{aligned}
\label{eq:advection_diffusion_fdm_sat}
\mbox{SAT}\;\;\, &= \mbox{SAT}_N+\mbox{SAT}_E+\mbox{SAT}_S+\mbox{SAT}_W,\\
\mbox{SAT}_N &= \tau_N\Big[\frac{1}{2}\left(\alpha_2\bm f_2(\bm u_N)-\Diag(\text{sign}(\bm u_N))\abs{\alpha_2\bm f_2(\bm u_N)}\right)\\
&\quad\quad+\frac{1}{2}\left(\tilde\alpha_2{\bm U}_N\bm f_2'(\bm u_N)-\Diag(\text{sign}(\bm u_N))\abs{\tilde\alpha_2{\bm U}_N\bm f_2'(\bm u_N)}\right)-\bm\Epsilon_N(\D_y\bm u)_N-\bm g_{bN}\Big]\otimes (\bm H_y^{-1}\bm e_m),\\
\mbox{SAT}_E\, &= \tau_E (\bm H_x^{-1} \bm e_n) \otimes \Big[\frac{1}{2}\left(\alpha_1\bm f_1(\bm u_E)-\Diag(\text{sign}(\bm u_E))\abs{\alpha_1\bm f_1(\bm u_E)}\right)\\
&\quad\quad+\frac{1}{2}\left(\tilde\alpha_1{\bm U}_E\bm f_1'(\bm u_E)-\Diag(\text{sign}(\bm u_E))\abs{\tilde\alpha_1{\bm U}_E\bm f_1'(\bm u_E)}\right)-\bm\Epsilon_E(\D_x\bm u)_E - \bm  g_{bE}\Big],\\
\mbox{SAT}_S\; &= \tau_S \Big[\frac{1}{2}\left(-\alpha_2\bm f_2(\bm u_S)-\Diag(\text{sign}(\bm u_S))\abs{\alpha_2\bm f_2(\bm u_S)}\right)\\
&\quad\quad+\frac{1}{2}\left(-\tilde\alpha_2{\bm U}_S\bm f_2'(\bm u_S)-\Diag(\text{sign}(\bm u_S))\abs{\tilde\alpha_2{\bm U}_S\bm f_2'(\bm u_S)}\right)+\bm\Epsilon_S(\D_y\bm u)_S -\bm g_{bS}\Big] \otimes (\bm H_y^{-1} \bm e_1),\\
\mbox{SAT}_W &= \tau_W (\bm H_x^{-1} \bm e_1) \otimes \Big[\frac{1}{2}\left(-\alpha_1\bm f_1(\bm u_W)-\Diag(\text{sign}(\bm u_W))\abs{\alpha_1\bm f_1(\bm u_W)}\right)\\
&\quad\quad+\frac{1}{2}\left(-\tilde\alpha_1{\bm U}_W\bm f_1'(\bm u_W)-\Diag(\text{sign}(\bm u_W))\abs{\tilde\alpha_1{\bm U}_W\bm f_1'(\bm u_W)}\right)+\bm\Epsilon_W(\D_x\bm u)_W - \bm  g_{bW}\Big],
\end{aligned}
\end{equation}
where $\bm g_{bN},\bm g_{bE},\bm g_{bS},\bm g_{bW}$ are vectors containing values of $g_b$ respectively evaluated at North, East, South, West boundary grid points; $\bm\Epsilon=\Diag(\bm\epsilon)$, where $\bm\epsilon$ is the discrete analogue of $\epsilon$ evaluated at the grid points, $\bm\Epsilon_N$, $\bm\Epsilon_E$, $\bm\Epsilon_S$, $\bm\Epsilon_W$ are parts of $\bm\Epsilon$ at the North, East, South, West boundaries; and
\begin{equation}
\label{eq:fdm_2d_operators2}
\hspace*{-0.5cm}
\begin{aligned}
\bm u_N &= (\bm I_n\otimes \bm e_m^T)\bm u, &\bm u_E&=(\bm e_n^T\otimes\bm I_m)\bm u, &\bm u_S&=(\bm I_n\otimes \bm e_1^T)\bm u, &\bm u_W&=(\bm e_1^T\otimes\bm  I_m)\bm u,\\
{\bm U}_N &= \Diag(\bm u_N), & {\bm U}_E &= \Diag(\bm u_E), & {\bm U}_S &= \Diag(\bm u_S), & {\bm U}_W &= \Diag(\bm u_W),\\
(\D_y \bm u)_N &= (\bm I_n \otimes \bm d_m) \bm u, &(\D_x \bm u)_E&= (\bm d_n \otimes\bm  I_m) \bm u, &(\D_y\bm u)_S&=(\bm I_n \otimes \bm d_1) \bm u, &(\D_x \bm u)_W&= (\bm d_1 \otimes\bm  I_m) \bm u.
\end{aligned}
\end{equation}
The following theorem completes the formulation of the FD approximation to \eqref{eq:advection_diffusion_fem}, \eqref{eq:advection_diffusion_fem_bc}.
\begin{theorem}
\label{theorem:advection_diffusion_bt_penalties}
By choosing $\tau_N = \tau_E = \tau_S = \tau_W = 1$, the approximation \eqref{eq:advection_diffusion_fdm_approx} with the outer boundary terms \eqref{eq:advection_diffusion_fdm_sat} is stable.
\end{theorem}
\begin{proof}
Let ${\bm U} = \Diag(\bm u)$. Deriving a discrete energy estimate from \eqref{eq:advection_diffusion_fdm_approx} and \eqref{eq:advection_diffusion_fdm_sat}, one \cblue{has}
\begin{align*}
&\;\frac{d}{dt} \norm{\bm u}_{\H}^2 + 2\bm u^T(\bm A\otimes \bm H_y)\bm u+ 2\bm u^T(\bm H_x\otimes \bm A)\bm u\\
=&\; -(\Diag(\text{sign}(\bm u)) \bm u)^T(\bm H_x\otimes \bm e_1\bm e_1^T+\bm e_m\bm e_m^T)\left(\abs{\alpha_2\bm{f}_2(\bm u)}+\abs{\hat\alpha_2{\bm U}\bm{f}_2'(\bm u)}\right)\\
&-(\Diag(\text{sign}(\bm u)) \bm u)^T(\bm e_1\bm e_1^T+\bm e_n\bm e_n^T\otimes \bm H_y)\left(\abs{\alpha_1\bm{f}_1(\bm u)}+\abs{\hat\alpha_1{\bm U}\bm{f}_1'(\bm u)}\right).
\end{align*}
\cpurple{It can be seen that the portion corresponding to each boundary is nonpositive. For example, the one corresponding to the North boundary is $-(\Diag(\text{sign}(\bm u)) \bm u)^T(\bm H_x\otimes \bm e_m\bm e_m^T)\left(\abs{\alpha_2\bm{f}_2(\bm u)}+\abs{\hat\alpha_2{\bm U}\bm{f}_2'(\bm u)}\right)$. Therein, the matrix $(\bm H_x\otimes \bm e_m\bm e_m^T)$ is positive definite. The vector $\abs{\alpha_2\bm{f}_2(\bm u)}+\abs{\hat\alpha_2{\bm U}\bm{f}_2'(\bm u)}$ is nonnegative. The vector $\Diag(\text{sign}(\bm u)) \bm u = \abs{\bm u}$ is also nonnegative. Therefore, the right-hand side of the above estimate is less than or equal to zero. Strict stability is thus guaranteed. The above estimate is analogous to the continuous energy estimate \eqref{eq:advection-diffusion-continuous-energy}.}
\end{proof}

\subsection{The FE scheme}
\label{sec:sbp_form_fem}
We formulate and prove the SBP properties of the FE methods in this section.
\subsubsection*{The Galerkin FE approximation}
Let $\{\mathcal T_h\}_{h>0}$ be a triangulation mesh of $\Omega$, and $\mathbb{P}^1(K)$ is the space of all linear functions on each element $K$ of $\{\mathcal T_h\}$. Denote $C^0(\Omega)$ the space of continuous functions on $\Omega$. We seek the solution of \eqref{eq:advection_diffusion_fem} in the space of weakly differentiable functions $H^1(\Omega)$. To this purpose, we define a scalar-valued Lagrange FE space
\[
V_h:= \{ v(\cdot,t) \in C^0(\Omega)\mid v|_K \in \mathbb{P}^1(K), \forall K \in \mathcal T_h \}.
\]
\cblue{Note that all the degrees of freedom of continuous Galerkin FE methods are located on the nodal points. More specifically, using $\mathbb{P}^1$ elements, they are the vertices of the triangulation $\mathcal T_h$}. A FE approximation of \eqref{eq:advection_diffusion_fem} subject to the boundary condition \eqref{eq:advection_diffusion_fem_bc} can then be formulated as: find $u_h \in V_h$ such that
\begin{equation}
\label{eq:gfem_1}
\begin{aligned}
(\partial_tu_h,v)+\cgreen{( \bm\alpha\grad \cdot \bm{f}(u_h) + \bm f'(u_h) \cdot \bm{\tilde\alpha}\grad u_h, v)} &\\
+(\epsilon\grad u_h, \grad v) =
&\;\frac{1}{2}\int_{\partial\Omega}\big((\bm\alpha\mathbf{n})\cdot \bm{f}(u_h)-\text{sign}(u_h)\abs{(\bm\alpha\mathbf{n})\cdot \bm{f}(u_h)}\big)vds\\
&+\frac{1}{2}\int_{\partial\Omega}\big((\bm{\tilde\alpha}\mathbf{n})\cdot (u_h\bm{f}'(u_h))-\text{sign}(u_h)\abs{(\bm{\tilde\alpha}\mathbf{n})\cdot (u_h\bm{f}'(u_h))}\big)v ds\\
&-\int_{\partial\Omega} g_bv ds, \;\forall v \in V_h.
\end{aligned}
\end{equation}
We can write $V_h = \mbox{span}\{\varphi_i\}_{i=1}^{\cblue{\mathrm N}}$, where $\varphi_i,i=1,\dots,\cblue{\mathrm N}$ are piecewise linear Lagrange basis functions. Inserting the ansatz $u_h = \displaystyle\sum_{j=1}^{\cblue{\mathrm N}}\xi_j(t)\varphi_j(x,y), \xi_i \in \R,j=1,2,\dots,\cblue{\mathrm N}$ and $v = \varphi_i,i=1,2,\dots,\cblue{\mathrm N}$ into \eqref{eq:gfem_1} gives
\begin{equation}
\label{eq:gfem_matrix}
\begin{aligned}
\mathbf{M} \dot{\bm{\xi}}(t) = -\mathbf{C}(u_h)\bm{\xi}(t) - \mathbf{A}\bm{\xi}(t)+\frac{1}{2}\mathbf{R}_M\left((\bm\alpha\mathbf{n})\cdot{\bm f}(\bm{\xi}(t))\right)-\frac{1}{2}\mathscr{R}_M(\text{sign}(u_h))\abs{(\bm\alpha\mathbf{n})\cdot{\bm f}(\bm{\xi}(t))}\\
+\frac{1}{2}\mathscr{R}_M\left((\bm{\tilde\alpha}\mathbf{n})\cdot \bm f'(u_h)\right)\bm{\xi}(t)-\frac{1}{2}\mathscr{R}_M\left(\abs{(\bm{\tilde\alpha}\mathbf{n})\cdot \bm f'(u_h)}\right)\bm{\xi}(t)-\bm{r},
\end{aligned}
\end{equation}
where the over dot denotes the time derivative; $\mathbf{M},\mathbf{C}(\phi),\mathbf{A},\mathbf{R}_M,\mathscr{R}_M(\phi),\bm{r}$ are respectively called the mass matrix, the convective flux matrix, the stiffness matrix, the Robin boundary mass matrix, the nonlinear Robin boundary mass matrix, and the Robin boundary vector. A different notation style is used for $\mathscr{R}_M(\phi)$ to indicate that each entry corresponds to a $\phi$-weighted line integral. The matrices are of full size \cblue{$\mathrm N\times \mathrm N$}.  To be more precise, for $i,j=1,\dots,\cblue{\mathrm N}$,
\begin{align*}
    \mathbf{M}_{ij} & = (\varphi_j,\varphi_i), &
    \mathbf{C}_{ij}(\phi) & = (\bm f'(\phi)\cdot\bm\alpha\grad\varphi_j,\varphi_i)+(\tilde{\bm\alpha}\grad\cdot (\bm f(\phi)\varphi_j),\varphi_i),&
    \mathbf{A}_{ij} & = (\epsilon\grad\varphi_j,\grad\varphi_i),\\
    \left(\mathbf{R}_M\right)_{ij} & = \int_{\partial\Omega}\varphi_j\varphi_ids, &
    \left(\mathscr{R}_M\right)_{ij}(\phi) & = \int_{\partial\Omega}\phi\varphi_j\varphi_ids, &
    \bm{r}_i & = \int_{\partial\Omega} g_{\cgreen b}\varphi_i ds. &
\end{align*}
\cblue{As in the standard continuous Galerkin method, given that $\phi, g_b, \epsilon$ present nodal weights, the above integrals can be evaluated exactly, see Remark \ref{remark:fe_quadrature}.} \cgreen{We write $\mathbf C$ instead of  $\mathbf C(\phi)$ when $\phi \equiv 1$.} Pseudo-code for the assembly of the above matrices can be found in \citet{Larson2009}, and \citet{Dao2019}. 

\cblue{
\begin{remark}\label{remark:fe_quadrature}
To clarify, we use the following quadrature rule to calculate the $\phi$-, $g_b$-, $\epsilon$-dependent matrices,
\[
\int_{K} \phi\varphi_j\varphi_idx\approx \frac{\phi(x_i)+\phi(x_j)}{2} \int_{K}\varphi_j\varphi_idx,
\]
for each element $K$ on $\mathcal T_h$, and
\begin{equation}\label{eq:quadrature_rule}
\int_{E} \phi\varphi_j\varphi_ids\approx \frac{\phi(x_i)+\phi(x_j)}{2} \int_{E}\varphi_j\varphi_ids,
\end{equation}
for each edge $E$ on the boundary $\Gamma$. The above integrals are exact if $\phi, g_b, \epsilon$ are considered as nodal weights. Later we will show that the stability proofs do not rely on these quadrature rules. Thus, any other consistent rules could have been used. The use of \eqref{eq:quadrature_rule} merely simplifies the SATs for coupling FE.
\end{remark}}

\begin{proposition}
Let $\grad \cdot \bm f$ be a linear advection flux, i.e., $\bm f'= \bm a = (a_1,a_2)^T\in\R^2$. Then, the advection operator $\mathbf{M}^{-1}\mathbf{C}$ and the diffusion operator $-\mathbf{M}^{-1}\mathbf{A} + \mathbf{M}^{-1}\mathbf{R}_C^{\partial \Omega},\mathbf{R}^{\partial \Omega}_{C\,ij} = \int_{\partial \Omega} \mathbf{n}\cdot \epsilon\grad\varphi_j \varphi_i ds$, $i,j=1,2,\dots,\cblue{\mathrm N}$ used in \eqref{eq:gfem_matrix} satisfy the summation-by-parts rules \eqref{eq:sbp_properties_first} and \eqref{eq:sbp_properties_second}, with the norm matrix $\mathbf{M}$.
\end{proposition}

\begin{proof}
We prove that the FE advection operator under the general notation $\mathcal{Q}=\mathbf{C}$ satisfies \eqref{eq:sbp_properties_first}. By integration-by-parts, we have $\mathbf{C}+\mathbf{C}^T=\mathscr{R}_M(\bm a\cdot \bm n)$. We need to show that $\mathscr{R}_M(\bm a\cdot \bm n)=\mathcal{L}^T \mathcal{B} \;\mbox{Diag}(\bm a\cdot\mathbf{n})\; \mathcal{L}$ with the FE formulation of $\mathcal{B}$ and $\mathcal{L}$. For a general boundary selection operator $\mathcal{L}$, each row of $\mathcal{L}$ must be a one-to-one mapping from its ``one'' element with a node on the boundary. The combination of a left multiplier $\mathcal{L}^T$ and a right multiplier $\mathcal{L}$ thus expands the local boundary operators into the full domain size (with zeros on the extended rows and columns). Relation \eqref{eq:sbp_properties_first} becomes obvious with the local boundary integration operator inserted $\mathcal{B}_{ij}\equiv\left(\mathbf{R}_M\right)_{ij}=\int_{\partial\Omega}\varphi_j\varphi_ids$, for $i,j=1,2,\dots,\cblue{\mathrm N_{ \partial \Omega}}$.

For the diffusion operator $\mathcal{R}=-\mathbf{A}+\mathbf{R}_C^W$, we need to show that $-\mathbf{A}+\mathbf{R}_C^{\partial \Omega}=\mathcal{L}^T\mathcal{B} \mathcal S-\mathcal{A}$ with the FE formulation of $\mathcal{B},\mathcal{L},\mathcal S$, and $\mathcal{H}$. Firstly, we identify $\widetilde{\mathcal S}(u_h) \in V_h, u_h \in V_h$ approximating $\mathbf{n}\cdot\epsilon\grad u_h$. The operator $\widetilde{\mathcal S}$ satisfies
\[
(\widetilde{\mathcal S}(u_h)-\mathbf{n}\cdot\epsilon\grad u_h,v)_{\partial\Omega} = 0,\, \forall v \in V_h.
\]
Since $\widetilde{\mathcal S}(u_h) \in V_h$, there exists $\bm{\eta}=(\eta_1,\eta_2,\dots,\eta_{\cblue{\mathrm N}})$ such that $\widetilde{\mathcal S}(u_h) = \sum_{j=1}^{\cblue{\mathrm N}} \eta_j \varphi_j$. Inserting $\widetilde{\mathcal S}(u_h) = \sum_{j=1}^{\cblue{\mathrm N}} \eta_j \varphi_j$ and $u_h=\sum_{j=1}^{\cblue{\mathrm N}}\xi_j\varphi_j$ into the above equation gives
\begin{equation}
\label{eq:fe_sbp_properties_3}
\tilde{\mathcal{B}}\bm\eta = \mathbf{R}_C^{\partial \Omega} \bm{\xi},
\end{equation}
where $\tilde{\mathcal{B}}_{ij}=\int_{\partial\Omega}\varphi_j\varphi_ids,i,j=1,2,\dots,\cblue{\mathrm N}$ and $\mathbf{R}_C^{\partial \Omega}$ is defined above. Since $\int_{\partial\Omega}\varphi_j\varphi_ids = 0$ if either $i$ or $j$ corresponds to a nonboundary node,
\begin{equation}
\label{eq:fe_sbp_properties_4}
\mathcal{L}^T\mathcal{B}\mathcal{L} = \tilde{\mathcal{B}}.
\end{equation}
The multiplication of $\mathcal{L}^T$ and $\mathcal{L}$ is because $\mathcal{B}$ is defined in the SBP framework of size \cblue{$\mathrm N_{\partial\Omega}\times\mathrm N_{\partial\Omega}$}. The matrix $\mathcal S$ is obtained by projecting $\widetilde{\mathcal S}$ onto the boundary, $\mathcal S \bm{\xi}= \mathcal{L}\bm\eta$. Note that $\mathcal S$ is not square because of the matrix dimensions in \eqref{eq:sbp_properties_second}. From \eqref{eq:fe_sbp_properties_3} and \eqref{eq:fe_sbp_properties_4}, we obtain
\begin{equation}
\label{eq:fe_sbp_properties_1}
\mathbf{R}_C^{\partial \Omega} = \mathcal{L}^T\mathcal{B} \mathcal S.
\end{equation}
What is left is to prove the positive semi-definiteness of $\mathbf{A}$. Let $\Psi(u,v):=(\epsilon\grad u,\grad v)_\Omega, u,v \in V$. Given $\bm{\nu}  = (\nu_1,\nu_2,\dots,\nu_\cblue{\mathrm N})^T\in\R^{\cblue{\mathrm N}}$ and $v_h = \sum_{j=1}^{\cblue{\mathrm N}}\nu_j\varphi_j$, we have
\begin{equation}
\label{eq:fe_sbp_properties_2}
\hspace{-.5cm}
\bm\nu^T \mathbf{A} \bm\nu = \sum_{i=1}^{\cblue{\mathrm N}} \sum_{j=1}^{\cblue{\mathrm N}} \nu_i \mathbf{A}_{ij} \nu_j =  \sum_{i=1}^{\cblue{\mathrm N}} \sum_{j=1}^{\cblue{\mathrm N}} \Psi(\nu_i \varphi_i, \nu_j\varphi_j) =  \Psi \left(\sum_{i=1}^{\cblue{\mathrm N}} \nu_i \varphi_i, \sum_{j=1}^{\cblue{\mathrm N}} \nu_j\varphi_j\right) = \Psi (v_h,v_h) = \norm{\grad v}^2_\epsilon \geq 0
\end{equation}
for all $v_h \in V_h$. Therefore, the desired property of $\mathcal{R}$ is achieved.\end{proof}

\begin{remark}
By the energy method, conservation property of the discrete fluxes \eqref{eq:advection_diffusion_fdm_approx} and \eqref{eq:gfem_matrix} approximating the continuous skew-symmetric flux \eqref{eq:skew_symmetric_assumption} is fulfilled under the following condition
\begin{equation}\label{eq:skew_symmetric_sufficient_discrete}
\begin{cases}
\alpha_1\bm f_1 - \tilde\alpha_1 \bm F_1' \bm u = 0,\\
\alpha_2\bm f_2 - \tilde\alpha_2 \bm F_2' \bm u = 0,
\end{cases}
\end{equation}
which naturally holds if the condition \eqref{eq:skew_symmetric_sufficient} of the flux $\bm f$ holds for every point in space.
\end{remark}

\section{Simultaneous-approximation-term technique for interface coupling}
\label{sec:coupling}
This section describes our proposed SAT technique for interface treatment. To make this paper self-contained, we derive the coupling of both FD--FD and FD--FE. The former one is not the main focus, and partially covered in the literature, see \citet{Lundquist2018,Mattsson2010}. \cgreen{Unlike the mentioned references which investigate linear problems, the technique presented in this paper is extended to nonlinear problems, as $\bm f$ can be nonlinear in \eqref{eq:advection_diffusion_fem}}. Description of compatible \textit{SBP-preserving interpolation operators} for nondiagonal norm schemes is included.
\subsection{Continuous analysis}
\label{sec:coupling_continuous}

For simplicity, the analysis is done on the coupling of two domains $\Omega_L$, $\Omega_R$ illustrated in Figure \ref{fig:coupling-interface}, where $u$ and $v$ present the parts of solution on the left and right domains, respectively,
\begin{equation}
\label{eq:advection_diffusion_coupling_continuous}
\begin{aligned}
\frac{\partial u}{\partial t} + \grad \cdot \bm{f}(u) & =  \grad \cdot (\epsilon \grad u), & (x,y) \in \Omega_L,\\
\frac{\partial v}{\partial t} + \grad \cdot \bm{f}(v) & =  \grad \cdot (\epsilon \grad v), & (x,y) \in \Omega_R.
\end{aligned}
\end{equation}

\begin{figure}[pos=h]
    \centering
    \includegraphics[scale=1]{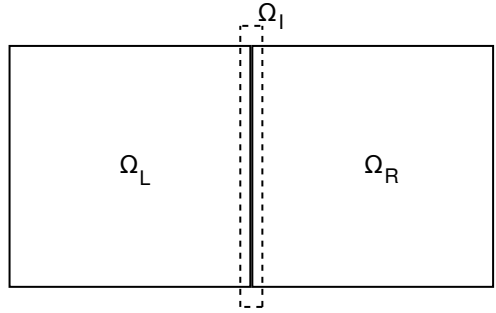}
    \caption{Coupling of two adjoining blocks.}
    \label{fig:coupling-interface}
\end{figure}

The North, South, West boundaries of $\Omega_L$ and the North, East, South boundaries of $\Omega_R$ will be referred to as the ``outer boundary'' and the shared edge of the two adjoining blocks \cpurple{is referred to} the ``interface'', denoted $\Omega_I$. \cblue{With the setup in Figure \ref{fig:coupling-interface},  $\Omega_I$ can be seen both as the East boundary of $\Omega_L$ and the West boundary of $\Omega_R$}. Applying the energy method on \eqref{eq:advection_diffusion_coupling_continuous}, a total energy estimate of the whole solution on $\Omega_L \cup \Omega_R$ reads
\[
\frac{d}{dt}\left(\norm{u}_{\Omega_L}^2+\norm{v}_{\Omega_R}^2\right) +2\norm{\grad u}^2_{\Omega_L,\epsilon} +2\epsilon\norm{\grad v}^2_{\Omega_R,\epsilon} = BT + IT,
\]
where $BT$ contains all the outer boundary terms being bounded with proper boundary treatment as discussed in Section \ref{sec:sbp_operators}, and $IT$ contains the interface terms given by
\begin{align*}
\hspace{-0.5cm}
IT & = \; -\int_{\Omega_I} \alpha_1 u f_1(u) ds+\int_{\Omega_I} \alpha_1 v f_1(v)ds -\int_{\Omega_I} \tilde\alpha_1 u^2 f_1'(u) ds+\int_{\Omega_I} \tilde\alpha_1 v^2 f_1'(v)ds + 2(u,\epsilon u_x)_{\Omega_I}-2(v,\epsilon v_x)_{\Omega_I}.
\end{align*}
Therefore, a sufficient condition for an accurate and energy-conserved coupling is
\begin{equation}
\label{eq:advection_diffusion_continuity}
\begin{cases}
    u = v,\\
    u_x = v_x
\end{cases}
\mbox{ at } x \in \Omega_I.
\end{equation}


\subsection{Necessary properties of interface interpolation for stability and conservation}
\label{sec:coupling_conditions}

Let $\bm u, \bm v$ be some discrete approximate solutions of $u, v$. An SBP-SAT semi-discretization of \eqref{eq:advection_diffusion_coupling_continuous} is given by
\begin{equation}
\label{eq:advection_diffusion_sbp_sat}
\begin{aligned}
\bm u_t + \mathcal{H}_L^{-1}\mathcal{Q}_L\bm\alpha\bm{f}(\bm u) + \tilde{\bm\alpha}\bm {F}'(\bm u)\mathcal{H}_L^{-1}\mathcal{Q}_L\bm u =  \mathcal{H}_L^{-1}\mathcal{R}_L^{(\epsilon)}\bm u + \mbox{SAT}_L+\mbox{SAT}_{IL},\\
\bm v_t + \mathcal{H}_R^{-1}\mathcal{Q}_R\bm\alpha\bm{f}(\bm v) + \tilde{\bm\alpha}\bm {F}'(\bm v)\mathcal{H}_R^{-1}\mathcal{Q}_R\bm v = \mathcal{H}_R^{-1}\mathcal{R}_R^{(\epsilon)}\bm v + \mbox{SAT}_R+\mbox{SAT}_{IR},
\end{aligned}
\end{equation}
where $\mbox{SAT}_{L},\mbox{SAT}_{R}$ present the weak boundary treatment, and $\mbox{SAT}_{IR},\mbox{SAT}_{IR}$ present the weak interface treatment. Another advantage of this weakly forcing technique is that $\mbox{SAT}_{R},\mbox{SAT}_{L},\mbox{SAT}_{IR},\mbox{SAT}_{IL}$ each can be treated separately by the energy method. Therefore, in the following subsections, we assume stable outer boundary treatment and only analyze the estimate terms regarding the additional involvement of $\mbox{SAT}_{IL},\mbox{SAT}_{IR}$.
\begin{figure}[pos=h]
    \centering
    \includegraphics[scale=1.2]{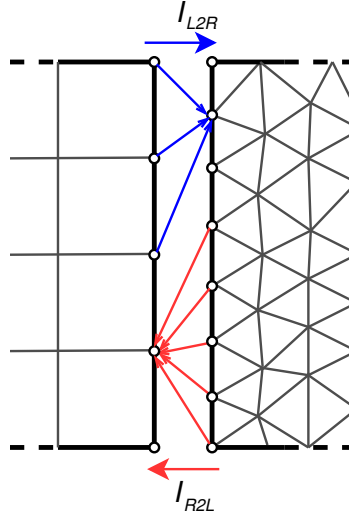}
    \caption{Example of translating nodal values on a nonmatching interface using interpolation operators.}
    \label{fig:nonmatching-interface}
\end{figure}
To prove stability and energy conservation, making use of the following class of interpolation operators to translate the shared-interface \cgreen{nodal values}, first introduced in \citet{Mattsson2010}, is necessary.
\cpurple{
\begin{definition}
\label{def:conservation_property_fd_fd}
Let $\mathcal{H}^L_y,\mathcal{H}^R_y$ be the integration operators over the interface of the left and the right domains, respectively. Let $\bm{I}_{L2R}$ and $\bm{I}_{R2L}$ be two $p$-th order accurate interpolation operators, where $\bm{I}_{L2R}$ maps the interface values from the left domain to the right domain, and $\bm{I}_{R2L}$ does the opposite direction,
\begin{equation}
\label{eq:accuracy_requirement}
\bm u_I = \bm{I}_{R2L} \bm v_I + \mathcal{O}(h^p), \bm v_I = \bm{I}_{L2R} \bm u_I + \mathcal{O}(h^p).
\end{equation}
We call $\left\{\bm{I}_{L2R}, \bm{I}_{R2L}\right\}$ a pair of \textit{SBP-preserving interpolation operators} if
\begin{equation}
\label{eq:conservation_property_fd_fd}
\mathcal{H}_y^R\bm{I}_{L2R} = \bm{I}_{R2L}^T\mathcal{H}_y^L,
\end{equation}
where $p = \min(p_L,p_R)$, and $p_L,p_R$ are the orders of accuracy in boundary closure of the two involving methods.
\end{definition}}
The matrix mappings $\bm{I}_{L2R}$ and $\bm{I}_{R2L}$ are illustrated in Figure \ref{fig:nonmatching-interface}. For FD, $\mathcal{H}_y$ is the $\bm H_y$-norm we defined earlier. For FV, this operator is a diagonal matrix containing the Euclid distances between two neighboring interface nodes, \citet{Lundquist2018}. For dG, this operator is a diagonal matrix presenting integration weights for interface edges, \citet{Kozdon2016}. However, it is well-known that in continuous Galerkin FE schemes, the one-dimensional norm matrix, i.e., the mass matrix, is not diagonal. In Appendix \ref{section:construct_interpolation_operators}, we present two techniques of constructing the pair $\left\{\bm{I}_{L2R},\bm{I}_{R2L}\right\}$ which satisfy and \eqref{eq:accuracy_requirement} and \eqref{eq:conservation_property_fd_fd}. The first technique Appendix \ref{section:construction_by_structure} is similar to
\citet{Mattsson2010}. \cgreen{Appendix \ref{section:shortcut_diagonal_norm} describes a way of acquiring the target interpolation operators through existing diagonal norm techniques.  This construction is beneficial due to the fact that various techniques to construct the necessary diagonal norm interpolation matrices can be found in the literature, e.g., \citet{Almquist2018,Kozdon2016,Lundquist2018,Lundquist2020,Mattsson2010}.}

\subsection{The FD--FD coupling}
We consider the case where FD schemes are used on both $\Omega_L$ and $\Omega_R$.  \cgreen{Let $\bm\omega \in \mathbb R^{\mathrm N}$ be an arbitrary FD solution. We denote by $\bm\omega_I\in\mathbb R^m$ the part of $\bm\omega$ on $\Omega_I$, where $m$ grid points are distributed on $\Omega_I$. For example, in the single-domain notations \eqref{eq:fdm_2d_operators2}, the orientation of $\Omega_L, \Omega_R$ gives that $\bm u_I \equiv \bm u_E, \bm v_I \equiv \bm v_W, (\D_x\bm u)_I \equiv (\D_x\bm u)_E, (\D_x\bm v)_I \equiv (\D_x\bm v)_W$}. The coupling SAT treatment is given by
\cblue{
\begin{equation}
\label{eq:advection_diffusion_sat_interface}
\hspace{-0.5cm}
\begin{aligned}
\mbox{SAT}_{IL} & = \mbox{SAT}_{IL}^{f} + \mbox{SAT}_{IL}^{f'} + \mbox{SAT}_{IL}^{\epsilon},\\
\mbox{SAT}_{IR} & = \mbox{SAT}_{IR}^{f} + \mbox{SAT}_{IR}^{f'} + \mbox{SAT}_{IR}^{\epsilon},
\end{aligned}
\end{equation}
where $\mbox{SAT}_{IL}^{f},  \mbox{SAT}_{IR}^{f}$ weakly couple the flux $f(u) = f (v)$ at $\Omega_I$
\begin{align*}
\mbox{SAT}_{IL}^{f} & = \alpha_L({\bm H_x^L}^{-1} \bm e_n)\otimes \left[(\alpha_1\bm f_1(\bm u))_I-\bm{I}_{R2L}(\alpha_1\bm f_1(\bm v))_I\right],\\
\mbox{SAT}_{IR}^{f} & = \alpha_R({\bm H_x^R}^{-1} \bm e_1) \otimes \left[(\alpha_1\bm f_1(\bm v))_I-\bm{I}_{L2R}(\alpha_1\bm f_1(\bm u))_I\right].
\end{align*}
The terms $\mbox{SAT}_{IL}^{f'}, \mbox{SAT}_{IR}^{f'}$ weakly couple $u=v$ at $\Omega_I$ with the weights of $\alpha_1 f'(u),\alpha_1 f'(v)$,
\begin{align*}
\mbox{SAT}_{IL}^{f'} & =\beta_L({\bm H_x^L}^{-1} \bm e_n)\otimes \left[(\tilde\alpha_1\bm F_1'(\bm u))_I\left(\bm u_I-\bm{I}_{R2L}\bm v_I\right)\right],\\
\mbox{SAT}_{IR}^{f'} & = \beta_R({\bm H_x^R}^{-1} \bm e_1) \otimes \left[(\tilde \alpha_1 \bm F_1'(\bm v))_I\left(\bm v_I-\bm{I}_{L2R}\bm u_I\right)\right],
\end{align*}
where $(\tilde\alpha_1\bm F_1'(\bm u))_I = \Diag(\tilde\alpha_1\bm f_1'(\bm u))_I, (\tilde \alpha_1 \bm F_1'(\bm v))_I = \Diag(\tilde\alpha_1\bm f_1'(\bm v))_I$. Finally, the terms $\mbox{SAT}_{IL}^{\epsilon}, \mbox{SAT}_{IR}^{\epsilon}$ impose the condition $u_x = u_y$  at $\Omega_I$
\begin{align*}
\mbox{SAT}_{IL}^{\epsilon} & = \delta_L ({\bm H_x^L}^{-1} \bm e_n) \otimes \left[\bm\Epsilon_{IL}(\D_x\bm u)_I - \bm{I}_{R2L}\bm\Epsilon_{IR}(\D_x\bm v)_I\right]\\
&+ \sigma_L ({\bm H_x^L}^{-1}\bm d_n^T)\otimes\left[\bm\Epsilon_{IL}(\bm u_I-\bm{I}_{R2L}\bm v_I)\right],\\
\mbox{SAT}_{IR}^{\epsilon} & = \delta_R ({\bm H_x^R}^{-1} \bm e_1) \otimes\left[\bm\Epsilon_{IR}(\D_xv)_I-\bm{I}_{L2R}\bm\Epsilon_{IR}(\D_x\bm u)_I\right]\\
&+ \sigma_R ({\bm H_x^R}^{-1} \bm d_1^T) \otimes\left[\bm\Epsilon_{IR}(\bm v_I-\bm{I}_{L2R}\bm u_I)\right],
\end{align*}
where  ${\bm\Epsilon}_{IL} = \Diag(\bm\epsilon_{IL}),{\bm\Epsilon}_{IR}= \Diag(\bm\epsilon_{IR})$; $\bm\epsilon_{IL}, \bm\epsilon_{IR}$ respectively contain the viscosity coefficients of the left and the right domain evaluated at the grid points on $\Omega_I$. The terms $\mbox{SAT}_{IL}^{\epsilon}, \mbox{SAT}_{IR}^{\epsilon}$ naturally allow interface jumps of the viscosity $\epsilon$, i.e., when $\bm\epsilon_{IL}\neq \bm\epsilon_{IR}$.} \cgreen{This property can be useful in applications involving discontinuous media, or when a significant difference in the amount of artificial viscosity is needed to stabilize the coupled schemes.}\\

\begin{remark}
Though $\bm e_1,\bm e_n,\bm e_m,\bm d_1,\bm d_n,\bm d_m$ all differ in the $x-$ and $y-$ dimensions, or in the left and right domains, the same notations are used to increase readability.
\end{remark}

\begin{theorem}
\label{theorem:choice_linear_fd_fd}
By choosing $\cblue{\alpha_L=\beta_L=1/2, \alpha_R=\beta_R=-1/2}, \delta_L=-\sigma_R, \sigma_L =-\delta_R$ and $\delta_L+\sigma_L = -1$, assuming proper outer boundary treatment, the approximation \eqref{eq:advection_diffusion_sbp_sat} with the interface SATs \eqref{eq:advection_diffusion_sat_interface} is stable.
\end{theorem}
\begin{proof}
Applying the energy method on each equation of \eqref{eq:advection_diffusion_sbp_sat} and adding them \cblue{together}, a total energy estimate can be written as
\begin{align*}
\frac{d}{dt} \norm{\bm u}_{\H}^2 + 2\bm u^T(\bm{A}\otimes \bm H_y^L)\bm u + 2\bm u^T(\bm H_x^L\otimes \bm{A})\bm u&\\+ \frac{d}{dt} \norm{\bm v}_{\H}^2 + 2\bm v^T(\bm{A}\otimes \bm H_y^R)\bm v+ 2\bm v^T(\bm H_x^R\otimes \bm A)\bm v & = BT + IT,
\end{align*}
where the outer boundary terms all contained in $BT$ are assumed to be finely tuned as given in \eqref{eq:advection_diffusion_fdm_sat} and Theorem \ref{theorem:advection_diffusion_bt_penalties}; and the interface terms consist of
\cblue{
\[
IT = IT^f + IT^{f'} + IT^{\epsilon} + IT^f_{SAT} + IT^{f'}_{SAT} + IT^{\epsilon}_{SAT},
\]
where $IT^f, IT^{f'}$ estimate the energy growth contributed by the discrete operators approximating the flux $\grad\cdot \bm f(u)$; $IT^{\epsilon}$ concerns the energy growth by the operators approximating $\grad \cdot (\epsilon \grad u)$; the terms $IT^f_{SAT}, IT^{f'}_{SAT}, IT^{\epsilon}_{SAT}$ are the contributions of the added SATs \eqref{eq:advection_diffusion_sat_interface}. Concretely, the terms $\mbox{SAT}_{IL}^{f},  \mbox{SAT}_{IR}^{f}$, resulting in the estimate $IT^f_{SAT}$, are used to control $IT^f$,
\begin{align*}
IT^f = & -\bm u^T(\bm e_n\bm e_n^T\otimes \bm H_y^L)(\alpha_1\bm f_1(\bm u))+\bm v^T(\bm e_1\bm e_1^T\otimes \bm H_y^R)(\alpha_1\bm f_1(\bm v)),\\
IT^{f}_{SAT} =&\;2\alpha_L\bm u^T(\bm e_n\bm e_n^T\otimes \bm H_y^L)\alpha_1\bm f_1(\bm u)-2\alpha_L\bm u^T(\bm e_n\bm e_1^T\otimes \bm H_y^L\bm{I}_{R2L})\alpha_1\bm f_1(\bm v)\\
&+2\alpha_R\bm v^T(\bm e_1\bm e_1^T\otimes \bm H_y^R)\alpha_1\bm f_1(\bm v)-2\alpha_R\bm v^T(\bm e_1\bm e_n^T\otimes \bm H_y^R\bm{I}_{L2R})\alpha_1\bm f_1(\bm u).
\end{align*}
The estimate $IT^f$ is derived from the SBP property in Definition \ref{sec:sbp_form_fdm}. By substituting $\alpha_L=1/2, \alpha_R=-1/2$, one can see that $IT^f$ is canceled by $IT^f_{SAT}$.  Accordingly, the terms $\mbox{SAT}_{IL}^{f'},  \mbox{SAT}_{IR}^{f'}$ are used to control $IT^{f'}$
\begin{align*}
IT^{f'} = & -\bm u^T(\bm e_n\bm e_n^T\otimes \bm H_y^L)(\bm U\tilde\alpha_1\bm f_1'(\bm u))+\bm v^T(\bm e_1\bm e_1^T\otimes \bm H_y^R)(\bm V\tilde\alpha_1\bm f_1'(\bm v)),\\
IT^{f'}_{SAT} =&\;2\beta_L\bm u^T(\bm e_n\bm e_n^T\otimes \bm H_y^L)(\bm U\tilde\alpha_1\bm f_1'(\bm u))-2\beta_L\bm v^T(\bm e_1\bm e_n^T\otimes \bm{I}_{R2L}^T\bm H_y^L)(\bm U\tilde\alpha_1\bm f_1'(\bm u))\\
&+2\beta_R\bm v^T(\bm e_1\bm e_1^T\otimes \bm H_y^R)(\bm V\tilde\alpha_1\bm f_1'(\bm v))-2\beta_R\bm u^T(\bm e_n\bm e_1^T\otimes \bm{I}_{L2R}^T\bm H_y^R)(\bm V\tilde\alpha_1\bm f_1'(\bm v)),
\end{align*}
in which $IT^{f'}$ is also canceled by $IT^{f'}_{SAT}$ with $\beta_L=1/2, \beta_R=-1/2$. The remaining mixed terms containing $\bm{I}_{L2R},  \bm{I}_{R2L}$ are
\begin{equation}\label{eq:thm_choice_fd_fd_mixed_terms}
\hspace{-0.5cm}
\begin{aligned}
IT^f  + IT^{f'} + IT^f_{SAT} + IT^{f'}_{SAT} =  & -2\alpha_L\bm u^T(\bm e_n\bm e_1^T\otimes \bm H_y^L\bm{I}_{R2L})\alpha_1\bm f_1(\bm v) -2\beta_R\bm u^T(\bm e_n\bm e_1^T\otimes \bm{I}_{L2R}^T\bm H_y^R)(\bm V\tilde\alpha_1\bm f_1'(\bm v))\\
 &-2\beta_L\bm v^T(\bm e_1\bm e_n^T\otimes \bm{I}_{R2L}^T\bm H_y^L)(\bm U\tilde\alpha_1\bm f_1'(\bm u))-2\alpha_R\bm v^T(\bm e_1\bm e_n^T\otimes \bm H_y^R\bm{I}_{L2R})\alpha_1\bm f_1(\bm u).
\end{aligned}
\end{equation}
The SBP property \eqref{eq:conservation_property_fd_fd} implies that $\bm{I}_{L2R}^T\bm H_y^R = \bm H_y^L\bm{I}_{R2L}$. Moreover, we have $\alpha_1 \bm f_1(\bm u) = \bm U\tilde\alpha_1\bm f_1'(\bm u)$ and  $\alpha_1 \bm f_1(\bm v) = \bm V\tilde\alpha_1\bm f_1'(\bm v)$ by the discrete skew-symmetry property \eqref{eq:skew_symmetric_sufficient_discrete}. The above choice of the penalty parameters makes $\alpha_L = - \beta_R$ and $\beta_L = - \alpha_R$. Therefore, the sum of two terms in the first line of \eqref{eq:thm_choice_fd_fd_mixed_terms} is zero. The sum of the two remaining terms in the second line is also zero. $SAT^f_{IL}$ and $SAT^{f'}_{IR}$ together imply conservation to the interface interpolation due to the discrete skew-symmetry \eqref{eq:skew_symmetric_sufficient_discrete}. The same can be said for the other pair $SAT^f_{IR}$ and $SAT^{f'}_{IL}$. Similarly, the estimate $IT^{\epsilon} _{SAT}$  of $\mbox{SAT}_{IL}^{\epsilon} ,  \mbox{SAT}_{IR}^{\epsilon} $ are used to control $IT^{\epsilon}$
\begin{align*}
IT^{\epsilon} = & \; \bm u^T(\bm\Epsilon_L\bm e_n\bm d_n+(\bm\Epsilon_L\bm e_n\bm d_n)^T\otimes \bm H_y^L)\bm u+ \bm v^T(-\bm\Epsilon_R\bm e_1\bm d_1-(\bm\Epsilon_R\bm e_1\bm d_1)^T\otimes \bm H_y^R)\bm v\\
IT^{\epsilon}_{SAT} = & \;\delta_L\left[\bm u^T(\bm\Epsilon_L\bm e_n\bm d_n+(\bm\Epsilon_L\bm e_n\bm d_n)^T\otimes \bm H_y^L)\bm u-\bm u^T(\bm\Epsilon_R\bm e_n\bm d_1+(\bm\Epsilon_R\bm e_n\bm d_1)^T\otimes \bm H_y^L\bm{I}_{R2L})\bm v\right]\\
& +\sigma_L\left[\bm u^T(\bm\Epsilon_L\bm e_n\bm d_n+(\bm\Epsilon_L\bm e_n\bm d_n)^T\otimes \bm H_y^L)\bm u-\bm u^T(\bm\Epsilon_L\bm e_1\bm d_n+(\bm\Epsilon_L\bm e_1\bm d_n)^T\otimes \bm H_y^L\bm{I}_{R2L})\bm v\right]\\
& +\delta_R\left[\bm v^T(\bm\Epsilon_R\bm e_1\bm d_1+(\bm\Epsilon_R\bm e_1\bm d_1)^T\otimes \bm H_y^R)\bm v-\bm v^T(\bm\Epsilon_L\bm e_1\bm d_n+(\bm\Epsilon_L\bm e_1\bm d_n)^T\otimes \bm H_y^R\bm{I}_{L2R})\bm u\right]\\
& +\sigma_R\left[\bm v^T(\bm\Epsilon_R\bm e_1\bm d_1+(\bm\Epsilon_R\bm e_1\bm d_1)^T\otimes \bm H_y^R)\bm v-\bm v^T(\bm\Epsilon_R\bm e_n\bm d_1+(\bm\Epsilon_R\bm e_n\bm d_1)^T\otimes \bm H_y^R\bm{I}_{L2R})\bm u\right].
\end{align*}
One can see that the mixed terms, e.g., $\delta_L\bm u^T(\bm\Epsilon_R\bm e_n\bm d_1+(\bm\Epsilon_R\bm e_n\bm d_1)^T\otimes \bm H_y^L\bm{I}_{R2L})$ and $\sigma_R \bm v^T(\bm\Epsilon_R\bm e_n\bm d_1+(\bm\Epsilon_R\bm e_n\bm d_1)^T\otimes \bm H_y^R\bm{I}_{L2R})\bm u$, are canceled as long as we choose $\delta_L=-\sigma_R, \sigma_L =-\delta_R$. To cancel $IT^{\epsilon}$, we also need $\delta_L+\sigma_L = -1$. Therefore, it is sufficient to conclude that with the above choice of parameters, all the interface terms vanish $IT \equiv 0$. The interface treatment is thus stable and conservative.}
\end{proof}

\subsection{The FD--FE coupling}
\label{sec:coupling_fd_fe}
This section contains the main contribution of this paper. \cpurple{The following definitions align FE schemes with the coupling to FD schemes.}

\begin{definition}
\label{def:fem_boundary_selection_operators}
\cpurple{Let $m$ be the number of degrees of freedom on the concerned part of the boundary, and $\bm L$ be a binary matrix of size $m\times \cblue{\mathrm N}$. The matrix $\bm L$ is called a \textit{FE boundary selection operator} if   the following properties are satisfied:}
\begin{itemize}
\item[(i)] Each row of \cblue{$\bm L$} contains exactly one ``1'' element uniquely corresponding to a node on the concerned part of the boundary,
\item[(ii)] \cblue{$\bm L$} rearranges the boundary nodes to the same spatial order of the coupled solution.
\end{itemize}
\end{definition}

The following properties of the above boundary selection operators are useful for the stability proofs in this section.

\cblue{
\begin{property}
\label{property:L_W}
Let $\bm L$ be a FE boundary selection operator. The following statements hold
\begin{itemize}
\item[(i)] $\bm E\bm L = \bm F \implies \bm E =\bm F \bm L^T$,
\item[(ii)] $\bm F \bm L^T = \bm E \implies \bm F = \bm E\bm L$\; if\; $(\bm I_{\cblue{\mathrm N}}-\bm L^T\bm L) \bm F^T = 0$,
\end{itemize}
\end{property}
\noindent where $\bm E, \bm F$ are matrices of appropriate sizes.

\begin{proof}
We prove (i). From $\bm E\bm L=\bm F$, right multiplying both sides by $\bm L^T$ gives  $\bm E\bm L\bm L^T=\bm F\bm L^T$. Since each row of $\bm L$ contains exactly one ``1'' element corresponding to the same position on each column of $\bm L^T$, it is clear that $\bm L\bm L^T=\bm I_m$. We have the first statement proven.\\

Similarly, we prove (ii) by right multiplying both sides by $L$ to obtain $\bm F\bm L^T\bm L=\bm E\bm L$. However, $\bm L^T\bm L\neq \bm I_{\cblue{\mathrm N}}$. The additional constraint $(\bm I_{\cblue{\mathrm N}}-\bm L^T\bm L) F^T = 0$ is necessary for getting $\bm F \bm L^T = \bm E$. Assuming that $\bm F$ is a square matrix \cblue{of size $\mathrm N \times \mathrm N$}, this additional constraint can be understood as, all the columns of $\bm F$ of which indices corresponding the nonboundary nodes are zero.
\end{proof}}

\cblue{We consider the setup in Figure \ref{fig:coupling-interface} where FD is used on $\Omega_L$ and FE is used on $\Omega_R$. The interface $\Omega_I$ is thus both the West boundary of $\Omega_R$ and the East boundary of $\Omega_L$. Therefore, the FE operators activating on the West boundary, e.g., $\bm L_W$, and the FD operators activating on the East boundary, e.g., $(\bm e_n^T\otimes\bm  I_m)$, are frequently used in this section. The notations $\bm H_x,\bm H_y$ are used instead of $\bm H_x^L,\bm H_y^L$ since there is only one domain involving FD}. \cred{Definition \ref{def:conservation_property_fd_fd} is restated below in the FD--FE coupling context.}
\begin{definition}
\label{def:conservation_property_fd_fe}
$\{\bm{I}_{L2R}, \bm{I}_{R2L}\}$ is called a pair of \textit{SBP-preserving FD--FE interpolation operators} if
\begin{equation}
\label{eq:conservation_property_fd_fe}
\mathbf{M}_I\bm{I}_{L2R} = \bm{I}_{R2L}^T\bm H_y,
\end{equation}
where $\displaystyle \mathbf{M}_I=\bm L_W\mathbf{R}_M^W\bm L_W^T$. Recall that $\cblue{\mathrm N}\times \cblue{\mathrm N}$ matrix $(\mathbf{R}_{M}^W)_{ij}=\int_{\Gamma_W}\varphi_j\varphi_ids,i,j=1,2,\dots,\cblue{\mathrm N}$.
\end{definition}
\cred{Notice that $\mathbf{M}_I$ is analogous to the norm $\bm H_y^R$ in the FD--FD coupling}. Despite it not being diagonal, $\mathbf{M}_I$ has nice properties including being symmetric and positive definite. For example, if the node distribution on $\Gamma_W$ is equidistant with step length $\bm H_y$, and $\mathbb{P}^1$ is used, then $\mathbf{M}_I$ is a banded matrix,
\[
\mathbf{M}_I = h_y\begin{bmatrix}
1/3 & 1/6 & 0 & 0 & \hdots & 0 \\
1/6 & 2/3 & 1/6 & 0 & \hdots & 0 \\
0 & 1/6 & 2/3 & 1/6 & \hdots & 0 \\
\vdots & 0 & \ddots & \ddots & \ddots & \vdots\\
0 & \hdots & 0 & 1/6 & 2/3 & 1/6\\
0 & \hdots & 0 & 0 & 1/6 & 1/3
\end{bmatrix}.
\]

\cgreen{Let $\bm\omega$  be a discrete vector containing nodal values of a FE solution in $V_h$. We denote $\bm \omega_I = \bm L_W \bm \omega$ the vector contains values of $\bm\omega$ at the nodes on $\Omega_I$}. We prove that the following FD--FE coupling SATs impose stability to \eqref{eq:advection_diffusion_sbp_sat},
\cblue{
\begin{equation}
\label{eq:advection_diffusion_fe_interface}
\begin{aligned}
\mbox{SAT}_{IL} & = \mbox{SAT}^{f}_{IL}  + \mbox{SAT}^{f'}_{IL} + \mbox{SAT}^{\epsilon}_{IL},\\
\mbox{SAT}_{IR} & = \mbox{SAT}^{f}_{IR} + \mbox{SAT}^{f'}_{IR} + \mbox{SAT}^{\epsilon}_{IR}.
\end{aligned}
\end{equation}
The terms $\mbox{SAT}^{f}_{IL}, \mbox{SAT}^{f'}_{IL}, \mbox{SAT}^{f}_{IR}, \mbox{SAT}^{f'}_{IR}$ couple the flux
\[
\begin{aligned}
\mbox{SAT}^{f}_{IL} & = \alpha_L ({\bm H_x^L}^{-1} \bm e_n)\otimes \left[(\alpha_1\bm f_1(\bm u))_I-\bm{I}_{R2L}(\alpha_1\bm f_1(\bm v))_I\right],\\
\mbox{SAT}^{f'}_{IL} & + \beta_L ({\bm H_x^L}^{-1} \bm e_n)\otimes \left[(\tilde\alpha_1\bm F_1'(\bm u))_I(\bm u_I-\bm{I}_{R2L}\bm v_I)\right],
\end{aligned}
\]
\vspace{-0.3cm}
\begin{equation}\label{eq:advection_diffusion_fe_interface_general}
\begin{aligned}
\mbox{SAT}^{f}_{IR}  & =\alpha_R \mathbf{M}^{-1}\mathbf{R}_M^W\bm L_W^T\left[(\alpha_1\bm f_1(\bm v))_I-\bm{I}_{L2R}(\alpha_1\bm f_1(\bm u))_I\right],\\
\mbox{SAT}^{f'}_{IR}&+ \beta_R\mathbf{M}^{-1}\Big[\mathscr{R}_M^W\left( \tilde\alpha_1\bm f_1'(\bm v)\right)\bm v-\bm V\mathbf{R}_M^W\bm L_W^T(\tilde\alpha_1\bm F_1'(\bm v))_I\bm{I}_{L2R}\bm u_I\Big],
\end{aligned}
\end{equation}
and the terms $\mbox{SAT}^{\epsilon}_{IL}, \mbox{SAT}^{\epsilon}_{IR}$ couple the diffusion term}
\[\cblue{
\begin{aligned}
\mbox{SAT}^{\epsilon}_{IL} & = \delta_L(H_x^{-1} \bm e_n) \otimes (\bm\Epsilon_L(\D_x\bm u)_I - \bm{I}_{R2L}\mathbf{D}_x^{W}(\epsilon)\bm v)+\sigma_L(H_x^{-1}\bm d_n^T)\otimes\bm\Epsilon_L(\bm u_I-\bm{I}_{R2L}\bm L_W\bm v),
\end{aligned}}
\]
\[\cblue{
\begin{aligned}
\mbox{SAT}^{\epsilon}_{IR} & = \delta_R\mathbf{M}^{-1}\mathbf{R}_M^W\bm L_W^T(\mathbf{D}_x^W(\epsilon)\bm v-\bm{I}_{L2R}\bm\Epsilon_L(\D_x\bm u)_I) +\sigma_R\mathbf{M}^{-1}(-\mathbf{R}^W_C(\epsilon))^T\bm L_W^T(\bm v_I - \bm{I}_{L2R}\bm u_I),
\end{aligned}}
\]
where $\mathbf{R}^W_{C\,ij}(\epsilon) = \int_{\Gamma_W} \mathbf{n}\cdot\epsilon\grad\varphi_j \varphi_i ds, \mathbf{D}_x^W(\epsilon) = -\mathbf{M}_I^{-1}\bm L_W\mathbf{R}^W_C(\epsilon)$. \cgreen{Other notations are defined in Section \ref{sec:sbp_form_fem}}.  The assembly of the above matrices is described in \citet{Larson2009,Dao2019}. Note that the minus signs coming with $\mathbf{R}_C^W(\epsilon)$ in $\mathbf{D}_x^W(\epsilon)$ and the third component of $\mbox{SAT}_{IR}$ are in order to make $\mathbf{R}_C^W(\epsilon)$ analogous to the FD operator $(\bm d_1\otimes\bm  I_m)$ since $\mathbf{n}_W$ is a negative vector.

\begin{lemma}
\label{lemma:R_M_W_nonlinear}
The following equality is fulfilled if the quadrature rule \eqref{eq:quadrature_rule} is used to calculate $\mathscr{R}_M(\phi)$,
\[
\mathscr{R}_M(\phi) = \frac{1}{2}\left(\mathbf{R}_M \mathbf \Phi +\mathbf \Phi\mathbf{R}_M\right),
\]
where $\mathbf \Phi$ is the diagonal matrix of which diagonal is $\left(\phi(x_2), \phi(x_2),\dots, \phi(x_{\mathrm N})\right)^T$.
\end{lemma}
The matrices $\mathscr{R}_M(\phi), \mathbf{R}_M$ are defined in Section \ref{sec:sbp_form_fem}. The result of Lemma \ref{lemma:R_M_W_nonlinear} is direct by writing \eqref{eq:quadrature_rule} in the matrix form. Restricting the integrals on $\Gamma_W$, we have 
\[
\mathscr{R}_M^W(\phi) = \frac{1}{2}\left(\mathbf{R}_M^W \mathbf \Phi +\mathbf \Phi\mathbf{R}_M^W\right).
\]

Using the quadrature rule \eqref{eq:quadrature_rule}, the terms $\mbox{SAT}^{f}_{IR}, \mbox{SAT}^{f'}_{IR}$ \eqref{eq:advection_diffusion_fe_interface_general} can then be simplified as
\begin{equation}\label{eq:advection_diffusion_fe_interface_simplified}
\begin{aligned}
\mbox{SAT}^{f}_{IR}  & = \alpha_R\mathbf{M}^{-1}\mathbf{R}_M^W\bm L_W^T\left[(\alpha_1\bm f_1(\bm v))_I-\bm{I}_{L2R}(\alpha_1\bm f_1(\bm u))_I\right],\\
\mbox{SAT}^{f'}_{IR}  &= \beta_R\mathbf{M}^{-1}(\tilde\alpha_1\bm F_1'(\bm v))\mathbf{R}_M^W\bm L_W^T\left[(\bm v_I-\bm{I}_{L2R}\bm u_I)\right],
\end{aligned}
\end{equation}
which is in the same form as the diagonal norm coupling \eqref{eq:advection_diffusion_sat_interface}. We show the relevance of the quadrature rule \eqref{eq:quadrature_rule} with the SAT \eqref{eq:advection_diffusion_fe_interface_simplified} in the Theorem \ref{theorem:advection_diffusion_fe_interface}.

Strict stability for the FD--FE coupling is shown in the theorem below.
\begin{theorem}
\label{theorem:advection_diffusion_fe_interface}
By choosing $\cblue{\alpha_L=\beta_L=1/2, \alpha_R=\beta_R=-1/2}, \delta_L=-\sigma_R, \sigma_L =-\delta_R$ and $\delta_L+\sigma_L = -1$, assuming proper outer boundary treatment, the approximation \eqref{eq:advection_diffusion_sbp_sat} with the interface SATs \eqref{eq:advection_diffusion_fe_interface} is stable.
\end{theorem}

\begin{proof}
A total energy estimate can be obtained by applying the energy method to both schemes
\begin{align*}
\frac{d}{dt} \norm{\bm u}_{\H}^2 + 2\epsilon \bm u^T(A\otimes H_y)\bm u + 2\epsilon \bm u^T(H_x\otimes A)\bm u+\frac{d}{dt}\norm{\bm v}^2_{\mathbf{M}}+2\epsilon \bm v^T\mathbf{A}\bm v & = BT + IT.
\end{align*}
\cblue{
The structure of this proof is similar to the proof of Theorem \ref{theorem:choice_linear_fd_fd}. We only focus on the key differences coming from FE.  The interface energy estimate reads
\[
IT = IT^{f} + IT^{f'} + IT^{\epsilon} + IT^f_{SAT} + IT^{f'}_{SAT} + IT^{\epsilon}_{SAT},
\]
where $IT^{f}, IT^{f'}$ come from the operators approximating the flux; $IT^{\epsilon}$ comes from the operators approximating the diffusion term;  $IT^f_{SAT}, IT^{f'}_{SAT}, IT^{\epsilon}_{SAT}$ come from the corresponding SATs. The following estimate is derived from the SBP property of the discrete operators
\begin{align*}
IT^{f} + IT^{f'} = & \; -\bm u^T(\bm e_n\bm e_n^T\otimes \bm H_y^L)(\alpha_1\bm f_1(\bm u)+\bm U\tilde\alpha_1\bm f_1'(\bm u))+\bm v^T\mathbf{R}_M^W\alpha_1 \bm f_1(\bm v)+\bm v^T\mathscr{R}_M^W(\tilde\alpha_1\bm f_1'(\bm v))\bm v.
\end{align*}
The involvement of $\mbox{SAT}^{f}_{IL}, \mbox{SAT}^{f'}_{IL}, \mbox{SAT}^{f}_{IR}, \mbox{SAT}^{f'}_{IR}$ \eqref{eq:advection_diffusion_fe_interface_simplified} mutually stabilizes the discrete flux
\begin{align*}
IT^{f}_{SAT} = &\;2\alpha_L\bm u^T(\bm e_n\bm e_n^T\otimes \bm H_y^L)\alpha_1\bm f_1(\bm u)-2\alpha_L\bm u^T(e_n\otimes H_y\bm{I}_{R2L}\bm L_W)\alpha_1\bm f_1(\bm v)\\
&+2\alpha_R\frac{(\alpha_1 \bm f_1(\bm v))^T\mathbf{R}_M^W\bm v+\bm v^T\mathbf{R}_M^W\alpha_1 \bm f_1(\bm v)}{2}+2\alpha_R\bm v^T(\bm e_n^T\otimes \mathbf{R}_M^W\bm L_W^T\bm{I}_{L2R})\alpha_1 \bm f_1(\bm u),\\
IT^{f'}_{SAT} = &\;2\beta_L\bm u^T(\bm e_n\bm e_n^T\otimes H_y)\tilde\alpha_1\bm f_1'(\bm u)-2\beta_L\bm v^T(e_n^T\otimes \bm L_W^T\bm{I}_{R2L}^TH_y){\bm U}\tilde\alpha_1\bm f_1'(\bm u)\\
&+2\beta_R\frac{(\bm V\tilde\alpha_1\bm f_1'(\bm v))^T\mathbf{R}_M^W\bm v+\bm v^T\mathbf{R}_M^W\bm V\tilde\alpha_1\bm f_1'(\bm v)}{2}-2\beta_R\bm u^T(\bm e_n\otimes \bm{I}_{L2R}^T\bm L_W\mathbf{R}_M^W)\bar v\tilde\alpha_1\bm f_1'(\bm v).
\end{align*}
Property \ref{property:L_W} is used to handle the transpose terms. Under Lemma \ref{lemma:R_M_W_nonlinear}, we have
\[
\bm v^T\mathscr{R}_M^W(\tilde\alpha_1\bm f_1'(\bm v)) \bm v=\bm v^T\frac{1}{2}\left(\mathbf{R}_M^W\tilde\alpha_1\bm F_1'(\bm v) + \tilde\alpha_1\bm F_1'(\bm v)\mathbf{R}_M^W  \right)\bm v = \frac{\bm v^T\mathbf{R}_M^W\bm V\tilde\alpha_1\bm f_1'(\bm v) + (\bm V\tilde\alpha_1\bm f_1'(\bm v))^T\mathbf{R}_M^W\bm v}{2}.
\]
Thus, the choice of $\alpha_L=\beta_L=1/2, \alpha_R=\beta_R=-1/2$  makes $-(IT^{f}+IT^{f'})$ appear in the sum $(IT^{f}_{SAT} + IT^{f'}_{SAT})$. Notice that Lemma \ref{lemma:R_M_W_nonlinear} is only necessary to show that the energy term $\bm v^T\mathscr{R}_M^W(\tilde\alpha_1\bm f_1'(\bm v)) \bm v$ coming from the discrete operator is canceled.  Because $\mathscr{R}_M^W$ is symmetric, the energy terms are canceled with the same choice of $\beta_L, \beta_R$ as long as the quadrature rule to compute $\mathscr{R}_M^W$ in the scheme and in the coupling SATs are the same. Similar to the proof of Theorem \ref{theorem:choice_linear_fd_fd}, the mixed terms containing the interpolation operators $\bm{I}_{L2R}, \bm{I}_{R2L}$ are also canceled because of the interpolation property \eqref{eq:conservation_property_fd_fe} and the discrete skew-symmetry \eqref{eq:skew_symmetric_sufficient_discrete}. The coupling of the diffusion terms can be proved to be stable in a similar manner.
\begin{align*}
IT^{\epsilon} = &\;\bm u^T(\bm\Epsilon_L\bm e_n\bm d_n+(\bm\Epsilon_L\bm e_n\bm d_n)^T\otimes H_y)\bm u + \bm v^T (\mathbf{R}^W_C + (\mathbf{R}^W_C)^T)\bm v,\\
 IT^{\epsilon}_{SAT} = &\; \delta_L \left[ \bm u^T(\bm\Epsilon_L\bm e_n\bm d_n+(\bm\Epsilon_L\bm e_n \bm d_n)^T \otimes H_y)\bm u -\bm u^T(\bm e_n \otimes H_y\bm{I}_{R2L}\mathbf{D}_x^W)\bm v-\bm v^T(\bm e_n^T\otimes (H_y\bm{I}_{R2L}\mathbf{D}_x^W)^T)u\right]\\
& + \sigma_L\left[\bm u^T(\bm\Epsilon_L\bm e_n\bm d_n+(\bm\Epsilon_L\bm e_n\bm d_n)^T\otimes H_y)\bm u-\bm u^T(\bm d_n^T\otimes H_y\bm{I}_{R2L}\bm\Epsilon_L\bm L_W)\bm v-\bm v^T(\bm d_n\otimes \bm L_W^T\bm\Epsilon_L\bm{I}_{R2L}^TH_y)\bm u\right]\\
& + \delta_R\left[-\bm v^T(\mathbf{R}_C^W+(\mathbf{R}_C^W)^T)\bm v-\bm v^T\mathbf{R}_M^W\bm L_W^T\bm{I}_{L2R}\bm\Epsilon_L(d_n\otimes\bm I_m)\bm u-\bm u^T(\bm d_n^T\otimes\bm  I_m)\bm\Epsilon_L\bm{I}_{L2R}^T\bm L_W\mathbf{R}_M^W\bm v)\right]\\
& - \sigma_R\bm v^T((\mathbf{R}^W_C)^T\bm L_W^T\bm L_W + \bm L_W^T\bm L_W\mathbf{R}^W_C)\bm v+\sigma_R \bm v^T(\mathbf{R}_C^W)^T\bm L_W^T\bm{I}_{L2R}(\bm e_n^T\otimes\bm  I_m)\bm u\\
& + \sigma_R \bm u^T(\bm e_n \otimes\bm I_m) \bm{I}_{L2R}^T\bm L_W\mathbf{R}_C^W\bm v.
\end{align*}
Due to the fact that $\mathbf{R}_C^W$ contains only zeros on the rows which do not correspond to the West boundary, we have $\bm L_W^T\bm L_W\mathbf{R}^W_C=\mathbf{R}^W_C$. As long as  $\delta_L=-\sigma_R, \sigma_L =-\delta_R$ and $\delta_L+\sigma_L = -1$ in $IT^{\epsilon}_{SAT}$, $IT^{\epsilon}$ is canceled by $IT^{\epsilon}_{SAT}$.  To show that the remaining terms in $IT^{\epsilon}_{SAT}$ containing $\bm{I}_{L2R}, \bm{I}_{R2L}$ also vanish, we need to show that
\begin{equation}\label{eq:advection_diffusion_fe_interface_proof_2}
H_y\bm{I}_{R2L}\mathbf{D}_x^W=-\bm{I}_{L2R}^T\bm L_W\mathbf{R}_C^W.
\end{equation}
Indeed, from the property \eqref{eq:conservation_property_fd_fe} which is $H_y\bm{I}_{R2L} = \bm{I}_{L2R}^T \mathbf{M}_I$, we have $H_y\bm{I}_{R2L} (-\mathbf{M}_I^{-1}\bm L_W \mathbf{R}_C^W) = -\bm{I}_{L2R}^T\bm L_W \mathbf{R}_C^W$, where $-\mathbf{M}_I^{-1}\bm L_W \mathbf{R}_C^W = \mathbf{D}_x^W$.
}
It it sufficient to conclude that the interface terms vanish $IT\equiv 0$ with the above choice of parameters.  \cblue{This result comes from the construction of the SATs}. The interface treatment is thus stable and conservative.
\end{proof}

\begin{remark}
Similar to the FD--FD coupling, the use of SBP-preserving interpolation operators used for FD--FE can also be shown to reserve the full order of accuracy. It is known that when narrow stencil FD operators of $q$th-order are used, the accuracy at boundary closures is of $q/2$th-order resulting in $(q/2+1)$th-order of overall accuracy. However, since the solution space $\mathbb{P}^1$ is currently used for FE, we expect that with roughly the same number of degrees of freedom of each method, the overall accuracy is bounded by second-order \cblue{even} if higher-order FD operators are used.
\end{remark}

\section{Numerical examples}
This section includes numerical experiments of the proposed coupling technique on the linear advection-diffusion equation, and a nonlinear viscous Burgers' equation.
\label{sec:numerical_results}

\subsection{Linear advection-diffusion equation}
The following linear equation is obtained by letting $\bm f = \bm a = (a_1,a_2)^T \in \R^2$ being a constant vector,
\begin{equation}\label{eq:linear_fem}
\begin{aligned}
u_t + \bm a \cdot \grad u & = \epsilon\grad\cdot(\grad u), && (x,y) \in \Omega,\\
u & = u_0(x,y), && (x,y) \in \Omega, t = 0.
\end{aligned}
\end{equation}
Inserting $\bm f = \bm a$ into \eqref{eq:skew_symmetric_precise} gives $\alpha_1=\alpha_2=\frac{1}{2}$. From \eqref{eq:advection_diffusion_fem_bc}, a set of well-posed boundary condition reads
\begin{equation}
\label{eq:linear_fem_bc}
\frac{1}{2} (\bm a\cdot\mathbf{n}-\abs{\bm a\cdot\mathbf{n}})u-\epsilon \mathbf{n}\cdot \grad u = g(x,y,t), \quad(x,y) \in \partial \Omega, t > 0.
\end{equation}
We consider the diffusion coefficient $\epsilon$ as a constant.

\subsubsection*{The FD--FD coupling}
A FD--FD semi-discretization of the coupling problem \eqref{eq:advection_diffusion_coupling_continuous} under the governing equation \eqref{eq:linear_fem} is given by
\begin{equation}
\label{eq:linear_sbp_sat}
\begin{aligned}
\bm u_t &= -a_1\D_x\bm u - a_2\D_y\bm u +\epsilon \D_{2x}\bm u + \epsilon \D_{2y}\bm u + \mbox{SAT}_L+\mbox{SAT}_{IL},\\
\bm v_t &= -a_1\D_x\bm v - a_2\D_y\bm v + \epsilon \D_{2x}\bm v + \epsilon \D_{2y}\bm v + \mbox{SAT}_R+\mbox{SAT}_{IR},
\end{aligned}
\end{equation}
where interface SATs are
\begin{equation}
\label{eq:linear_sat_interface}
\begin{aligned}
\mbox{SAT}_{IL} & = \frac{1}{2}a_1 ({\bm H_x^L}^{-1} \bm e_n)\otimes (\bm u_I-\bm I_{R2L}\bm v_I)\\
&+ \delta_L\epsilon ({\bm H_x^L}^{-1} \bm e_n) \otimes ((\D_x\bm u)_I - \bm I_{R2L}(\D_x\bm v)_I) + \sigma_L\epsilon ({\bm H_x^L}^{-1}\bm d_n^T)\otimes(\bm u_I-\bm I_{R2L}\bm v_I),\\
\mbox{SAT}_{IR} & = -\frac{1}{2}a_1({\bm H_x^R}^{-1}\bm e_1) \otimes (\bm v_I-\bm I_{L2R}\bm u_I)\\
&+\delta_R \epsilon({\bm H_x^R}^{-1}\bm e_1) \otimes ((\D_x\bm v)_I-\bm I_{L2R}(\D_x\bm u)_I)+ \sigma_R \epsilon({\bm H_x^R}^{-1} \bm d_1^T) \otimes (\bm v_I-\bm I_{L2R}\bm u_I),
\end{aligned}
\end{equation}
where $\delta_L=-\sigma_R, \sigma_L =-\delta_R, \delta_L+\sigma_L = -1$.
\subsubsection*{The FD--FE coupling}
A FD--FE semi-discretization of the coupling problem \eqref{eq:advection_diffusion_coupling_continuous} under the governing equation \eqref{eq:linear_fem} is given by
\begin{equation}
\label{eq:linear_fe_sbp_sat}
\begin{aligned}
\bm u_t &= -a_1\D_x\bm u - a_2\D_y\bm u +\epsilon \D_{2x}\bm u + \epsilon \D_{2y}\bm u + \mbox{SAT}_L+\mbox{SAT}_{IL},\\
\bm v_t &= -\mathbf{M}^{-1}\mathbf{C}\bm v-\epsilon \mathbf{M}^{-1}\mathbf{A}\bm v+  \frac{1}{2} (a_2-\abs{a_2})\mathbf R_M^N\bm v+  \frac{1}{2} (a_1-\abs{a_1})\mathbf R_M^E\bm v+\frac{1}{2} (-a_2-\abs{a_2})\mathbf R_M^S\bm v - \bm r +\mbox{SAT}_{IR},
\end{aligned}
\end{equation}
where $\mathbf{R}_M^N,\mathbf{R}_M^E,\mathbf{R}_M^S$ are obtained by limiting the integration of $\mathbf{R}_M$ to $\Gamma_N,\Gamma_E,\Gamma_S$, respectively. The coupling penalty terms are given by
\begin{equation}
\label{eq:advection_diffusion_sat_interface2}
\begin{aligned}
\mbox{SAT}_{IL} & = \frac{1}{2}a_1 (\bm H_x^{-1} \bm e_n)\otimes (\bm u_I-\bm{I}_{R2L}\bm L_W\bm v)\\
&+ \delta_L\epsilon(\bm H_x^{-1} \bm e_n) \otimes ((\D_x\bm u)_I - \bm{I}_{R2L}\mathbf{D}_x^W\bm v)+\sigma_L\epsilon(\bm H_x^{-1}\bm d_n^T)\otimes(\bm u_I-\bm{I}_{R2L}\bm L_W\bm v),\\
\mbox{SAT}_{IR} & =  -\frac{1}{2}a_1 \mathbf{M}^{-1}\mathbf{R}_M^W\bm L_W^T(\bm L_W\bm v - \bm{I}_{L2R}\bm u_I)\\
&+\delta_R\epsilon\mathbf{M}^{-1}\mathbf{R}_M^W\bm L_W^T(\mathbf{D}_x^W\bm v-\bm{I}_{L2R}(\D_x\bm u)_I) +\sigma_R\epsilon\mathbf{M}^{-1}(-\mathbf{R}^W_C)^T\bm L_W^T(\bm L_W\bm v - \bm{I}_{L2R}\bm u_I),
\end{aligned}
\end{equation}
where $\delta_L=-\sigma_R, \sigma_L =-\delta_R, \delta_L+\sigma_L = -1$.
\subsubsection{Eigenvalue analysis for validation of the interface treatment}
In this subsection, we show a numerical validation of Theorem \ref{theorem:advection_diffusion_fe_interface}. Consider the model problem \eqref{eq:advection_diffusion_coupling_continuous} subject to the boundary condition \eqref{eq:linear_fem_bc}. Second-order FD operators are used on $\Omega_L = [-1,1]\times[-2,0]$, and a FE scheme using piecewise linear basis functions ($\mathbb{P}^1$) is used on $\Omega_R = [-1,1]\times[0,2]$. The FD domain is discretized with $41^2$ grid points ($m=n=41$). The discretization of FE scheme uses 2796 vertices for an unstructured mesh using Delaunay triangulation with maximum cell diameter $h_{\max}=0.025$ and interface nodes matched to the FD grid, intended for an equally distributed image of eigenvalues. The solution is transported from the left to the right by setting $\bm a = (1,0)^T$.

Figure \ref{fig:fd-fe-2d} shows another example using $41^2$ grid points on the FD domain, and an unstructured mesh using Delaunay triangulation with maximum element diameter $h_{max}=0.2$ on the FE domain (nonmatching coupling with 41 FD grid points, and 51 FE nodal points on the interface). \cblue{The reference solution,
\begin{equation}\label{eq:analytic_inviscid_solution}
u(x,y) = r_1\exp\left(\frac{(x-x_0-a_1t)^2+(y-y_0-a_2t)^2}{r_2^2}\right),
\end{equation}
is initialized at $(x_0,y_0)^T = (-1,0)^T$ and transported from left to right by setting $\bm a = (1,0)^T$}. The diffusion parameter is set to zero $\epsilon = 0$. \cblue{The other parameters are $r_1 = 1, r_2 = 0.2$.} Sixth-order FD operators and $\mathbb{P}^1$ finite elements are used. The solution is examined at $t = 0$ and $t=1$.

\begin{figure}[pos=h]
    \begin{subfigure}{\textwidth}
    \centering
    \includegraphics[scale=0.2]{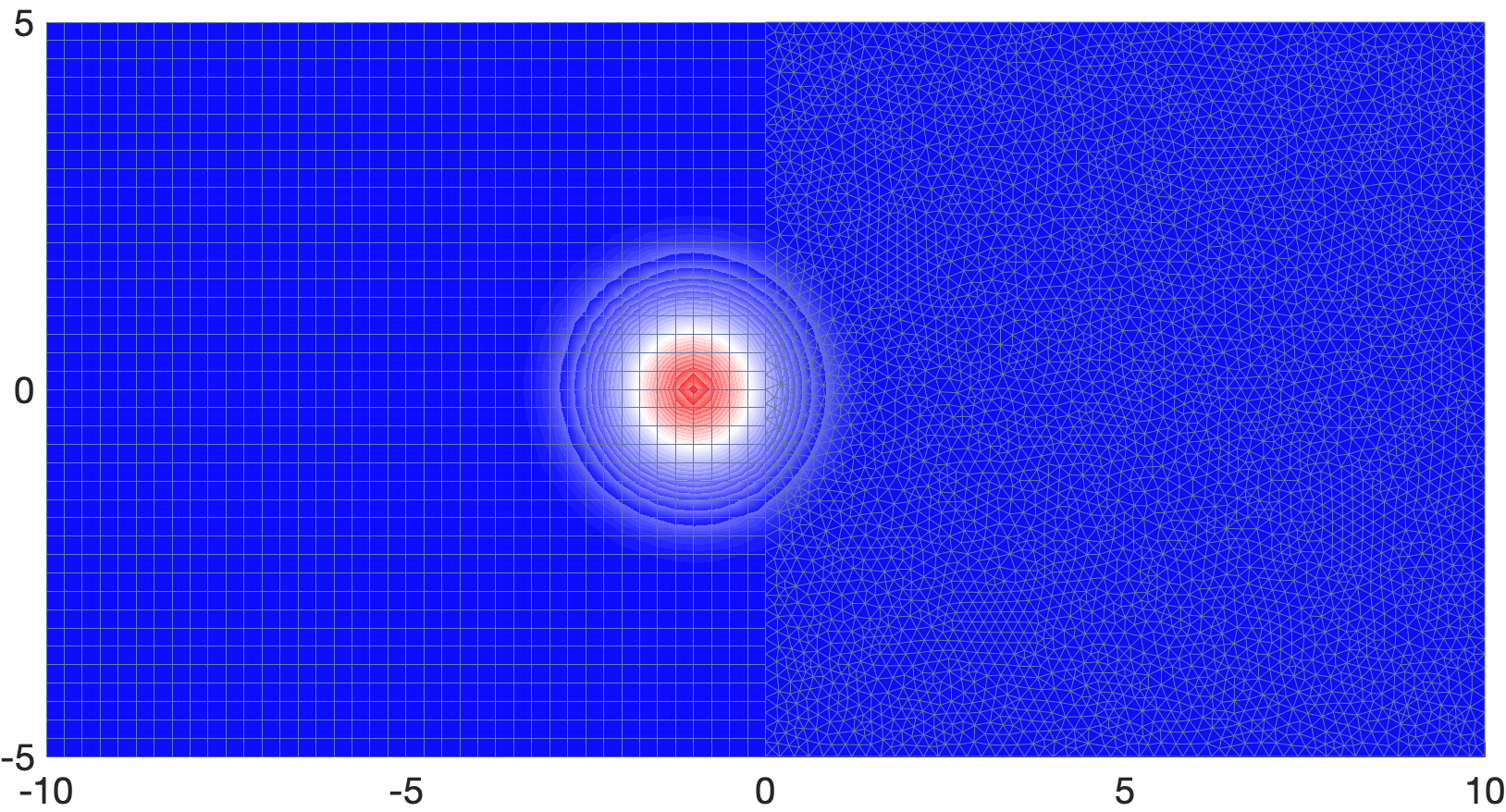}
    \caption{Initial data $t = 0$.}
    \end{subfigure}
    \begin{subfigure}{\textwidth}
    \centering
    \includegraphics[scale=0.2]{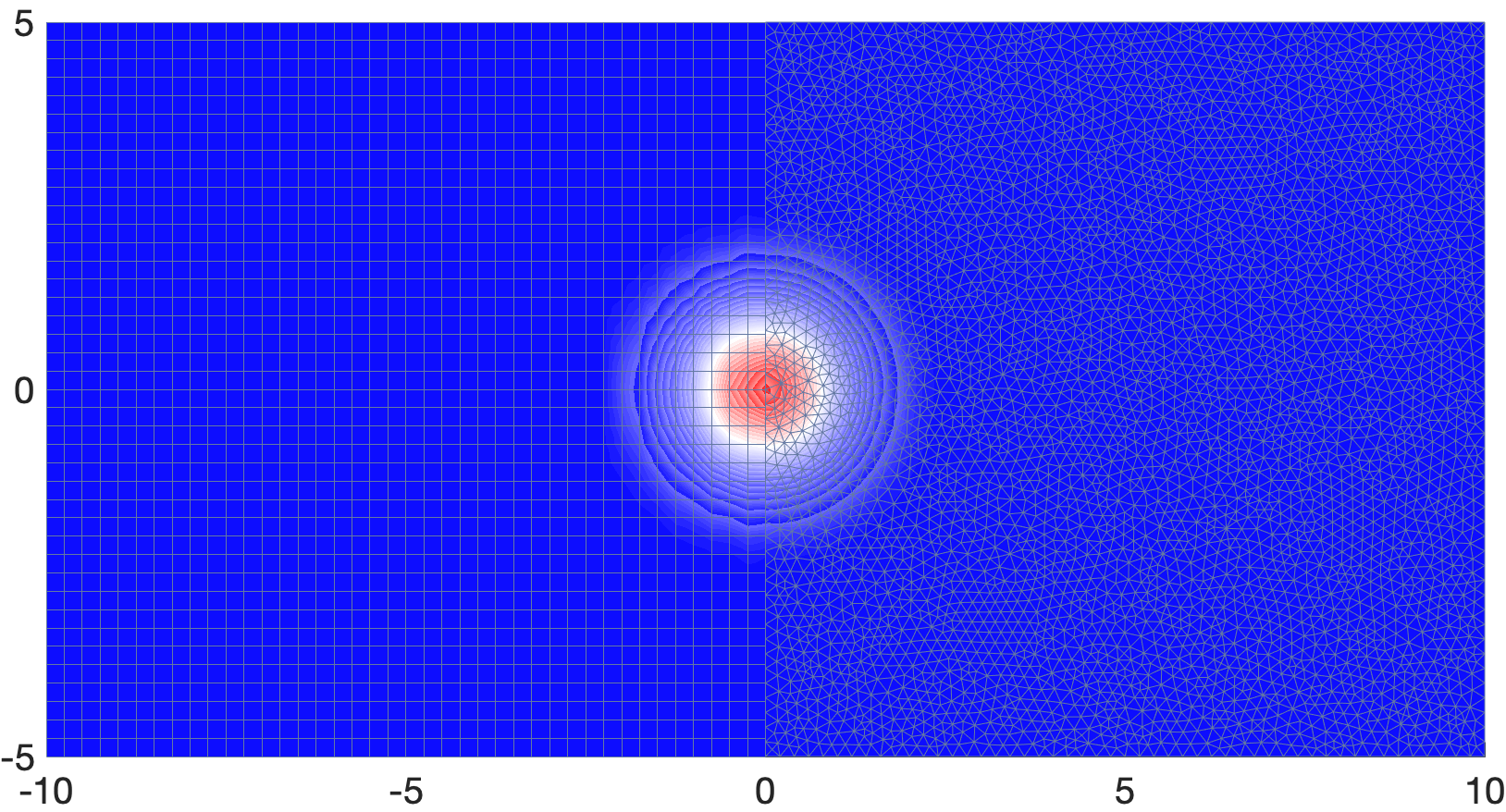}
    \caption{Computed solution at $t = 1$.}
    \end{subfigure}
    \caption{An example of FD--FE nonconforming multiblock coupling for the inviscid problem $\epsilon=0$ using a sixth-order FD scheme on the left domain (discretized with $41^2$ grid points) and a FE scheme on the right domain (using 8552 $\mathbb{P}^1$ elements). Notice the unmatching placement of nodal points on the interface.}
    \label{fig:fd-fe-2d}
\end{figure}

A semi-discrete problem to \eqref{eq:advection_diffusion_fem} can be written in the form
\cblue{
\begin{equation}
\label{eq:linear-semi-discrete-form}
\begin{bmatrix}
\bm u\\\bm v
\end{bmatrix}_t
= \mathbf{F} \begin{bmatrix}
\bm u\\\bm v
\end{bmatrix},
\end{equation}}
where $\mathbf{F}$ is a square real-valued matrix. In Figure \ref{fig:spectrum_linear_FD-FE}, the eigenvalues of $\mathbf F$ are plotted for several choices of the penalty parameters imposing the interface conditions in Theorem \ref{theorem:advection_diffusion_fe_interface}. \cblue{The horizontal axis indicates the real part of the eigenvalues. The imaginary axis is vertical.} Figure \ref{fig:spectrum_linear_FD-FE}(a) and \ref{fig:spectrum_linear_FD-FE}(c) are with two different sets of penalty parameters taking into account a diffusion term $\epsilon=0.01$ both satisfying Theorem \ref{theorem:advection_diffusion_fe_interface} while the parameters used in Figure \ref{fig:spectrum_linear_FD-FE}(b) and \ref{fig:spectrum_linear_FD-FE}(d) violate the theorem. It can be clearly seen that in the former cases, all the yielding eigenvalues have nonpositive real part. Therefore, with some explicit method to integrate in time, for instance Runge-Kutta 4, and a sufficiently small time step, the fully discrete scheme is stable.

\begin{figure}[pos=h]
    \centering
    \begin{subfigure}{0.49\textwidth}
    \includegraphics[scale=0.3]{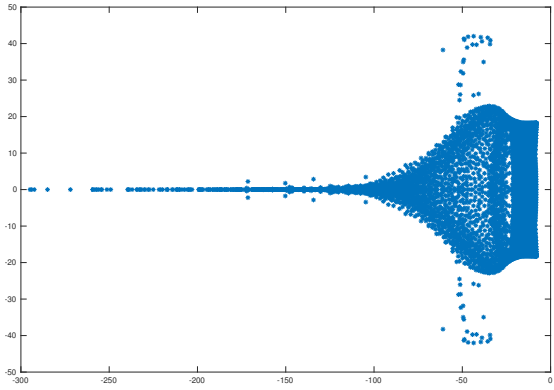}
    \caption{With a set of correct penalty parameters $\alpha_L = 0.5,\;\;\;\;\;$ $\alpha_R=-0.5, \delta_L=-\epsilon, \sigma_R=-\epsilon, \sigma_L=\delta_R=0$.}
    \end{subfigure}
    \begin{subfigure}{0.49\textwidth}
    \includegraphics[scale=0.3]{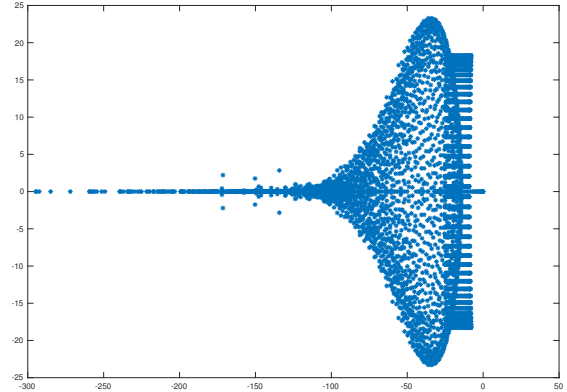}
    \caption{Without coupling terms $\alpha_{L,R}=\delta_{L,R}=\sigma_{L,R}=0$. There exist positive eigenvalues.}
    \end{subfigure}
    \begin{subfigure}{0.49\textwidth}
    \includegraphics[scale=0.3]{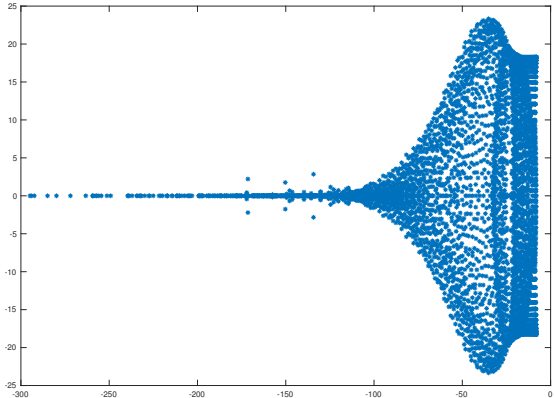}
    \caption{With a set of correct penalty parameters $\alpha_L = 0.5,\;\;\;\;\;$, $\alpha_R=-0.5, \delta_L=\sigma_L = -\epsilon/2, \delta_R=\sigma_R=\epsilon/2$.}
    \end{subfigure}
    \begin{subfigure}{0.49\textwidth}
    \includegraphics[scale=0.3]{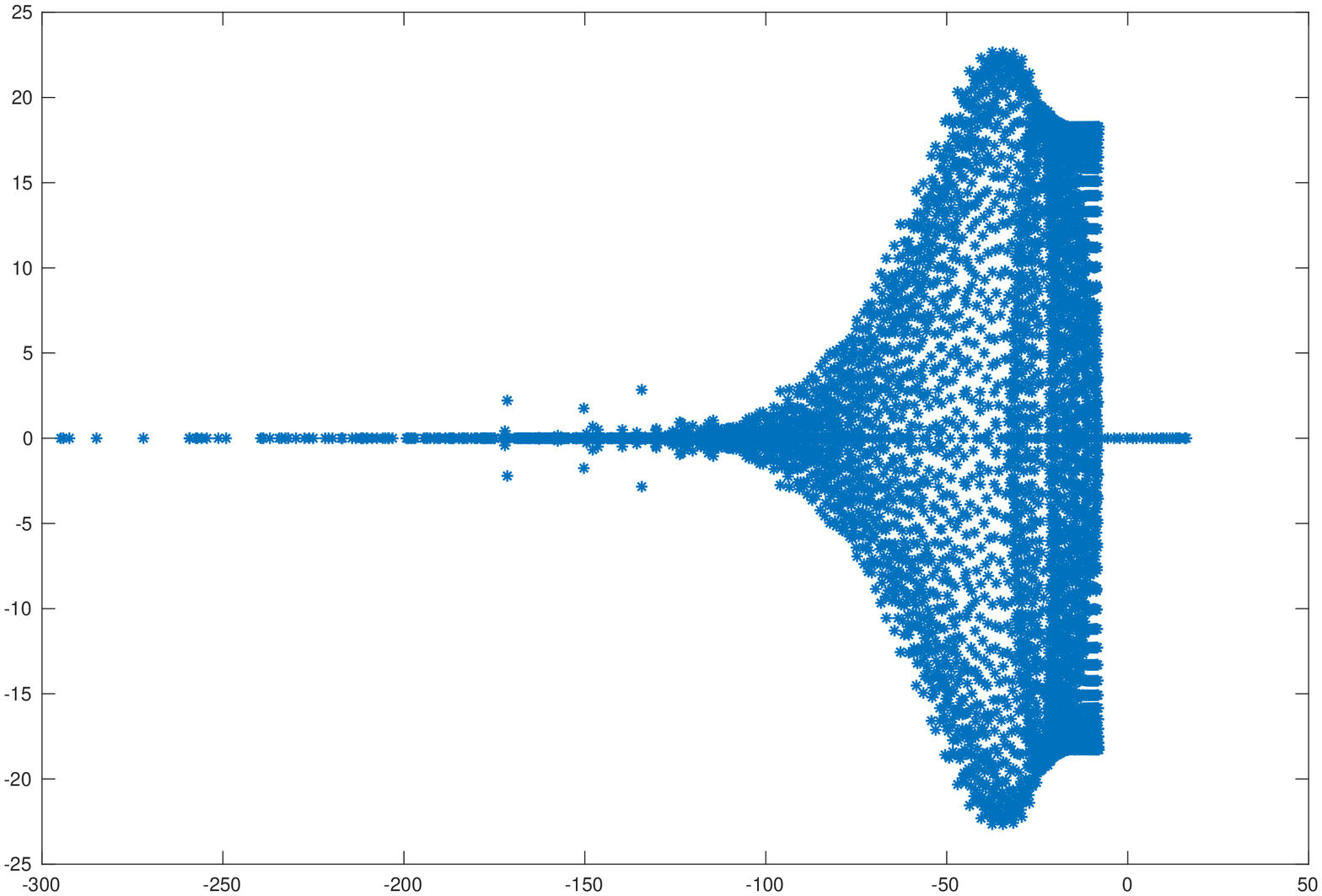}
    \caption{With violating parameters $\alpha_L = -0.05, \delta_L=0.001, \alpha_R=-0.05, \sigma_R=0.001, \sigma_L=\delta_R = 0$.}
     \end{subfigure}
    
    \caption{Eigenvalues \cblue{in the complex plane} of the system \eqref{eq:advection_diffusion_sbp_sat} coupling second-order FD -- $\mathbb{P}^1$ FE with parameters satisfying Theorem \ref{theorem:advection_diffusion_fe_interface} (on the left column) and violating the theorem (on the right column). All are with proper outer boundary treatment.}
    \label{fig:spectrum_linear_FD-FE}
\end{figure}

\subsubsection{Convergence study}
\label{sec:linear_convergence_study}

We use the following two dimensional Gaussian as reference solution
\[
u(x,y) = \frac{r}{4\pi(t-t_0)\epsilon} \exp\left[\frac{-1}{t-t_0}\left(\frac{(x-x_0-a_1t)^2+(y-y_0-a_2t)^2}{4\epsilon}\right)\right].
\]
The solution satisfies the boundary condition \eqref{eq:advection_diffusion_fem_bc} with zero boundary data $g=0$ if the initial solution is small enough at the inflow boundary. For convergence results in this section, we choose $r=0.005$. For the inviscid problem, we set $\epsilon = 0$. The 2D Gaussian \eqref{eq:analytic_inviscid_solution} is used as the reference solution. We set $r_1=1,r_2=0.2$. The $l_2$-error in each FD solution part $\bm u_{FD}$ and in each FE solution part $\bm u_{FD}$ are respectively calculated by
\begin{equation}
\label{eq:discrete_error_measurement}
l_2e(\bm u_{FD}) = \sqrt{(\bm u-\bm u_{e})^T\H(\bm u-\bm u_{e})},\quad l_2e(\bm u_{FE})=\sqrt{(\bm u-\bm u_{e})^T\mathbf{M}(\bm u-\bm u_{e})},
\end{equation}
where $\bm u_e$ is the exact solutions at the time of evaluation. The convergence rate in each method block is calculated as
\begin{equation}
\label{eq:convergence_rate}
Q = \log\left(\frac{l_2e(\bm u)^{(h_1)}}{l_2e(\bm u)^{(h_2)}}\right)/\log\left(\frac{h_1}{h_2}\right),
\end{equation}
for $h_1\neq h_2$ being two different mesh/grid sizes. Since multiple blocks with different methods are involved, the global convergence rate is calculated based on the sum of all $l_2$-errors, and the scale of mesh/grid sizes, e.g., $h_2/h_1=2$ when refining all the meshes/grids of size $h_1$ twice. We use the FD diagonal norm SBP operators formulated in \citet{Mattsson2004}. To demonstrate a proposed technique of constructing nondiagonal norm interpolation operators, the blocks are coupled by the `glue grid' FD-FD interpolation operators \citet{Kozdon2016}, and our proposed procedure Appendix \ref{section:shortcut_diagonal_norm}.

We choose the multiblock setup in Figure \ref{fig:2by2_domain}(a) to investigate convergence with an intention of examining possible degradation at corner points where the normal vector $\mathbf{n}$ is not defined. Tables \ref{table:2by2_multiblock_test1}, \ref{table:2by2_multiblock_test2}, \ref{table:2by2_multiblock_test3}, \ref{table:2by2_multiblock_test4} show the $l_2$-error and convergence rates for the problem \eqref{eq:linear_fem} subject to the boundary condition \eqref{eq:advection_diffusion_fem_bc} with different numbers of grid points. \eqref{eq:linear-semi-discrete-form} is integrated in time using Runge-Kutta 4 with time step $\Delta t = \min_{\text{blocks}}\{\beta h_{\max}^2\}$, where $\beta = 0.1$, and $h_{\max} = 1/(m-1)$ or $h_{\max}=r_{\max}$ the maximum element diameter depending on the method of the block. The solution is examined at $t=2$. The convergence rate is denoted by $Q$ and calculated by \eqref{eq:convergence_rate}. We respectively denote the number of FD grid points in each dimension $m$ (FD) and the maximum element diameter $r_{\max}$ (FE). For example, in the experiments where $m$ (FD) $=91$ and $r_{\max}$ (FE) $=D/91$, the discretization is carried out by $24843$ FD grid points and $27764$ Delaunay triangulated finite elements ($14063$ nodal points).

In Tables \ref{table:2by2_multiblock_test1} and \ref{table:2by2_multiblock_test2}, $m$ is equal to both grid size in the FD domain ($m=n_L=m_L$), and the number of discretization points along each axis of the FE mesh (matching interface). \cblue{In Table \ref{table:2by2_multiblock_test3} and \ref{table:2by2_multiblock_test4}, we perform the same study but when there are twice as many FE mesh points as FD grid points being located on the interfaces}. The obtained rates are expected between second-order FE approximation and fourth-order FD approximation. Order degradation by the edges/corner point is not observed. Simpler experiments involving equal distribution of the two methods can be found in \citet{Dao2019}. There one can see that the global convergence is bounded by the lowest second-order rate in $\mathbb P^1$ FE regions. 

\begin{figure}[pos=h]
    \centering
    \begin{subfigure}{0.33\textwidth}
    \centering
    \includegraphics[scale=0.5]{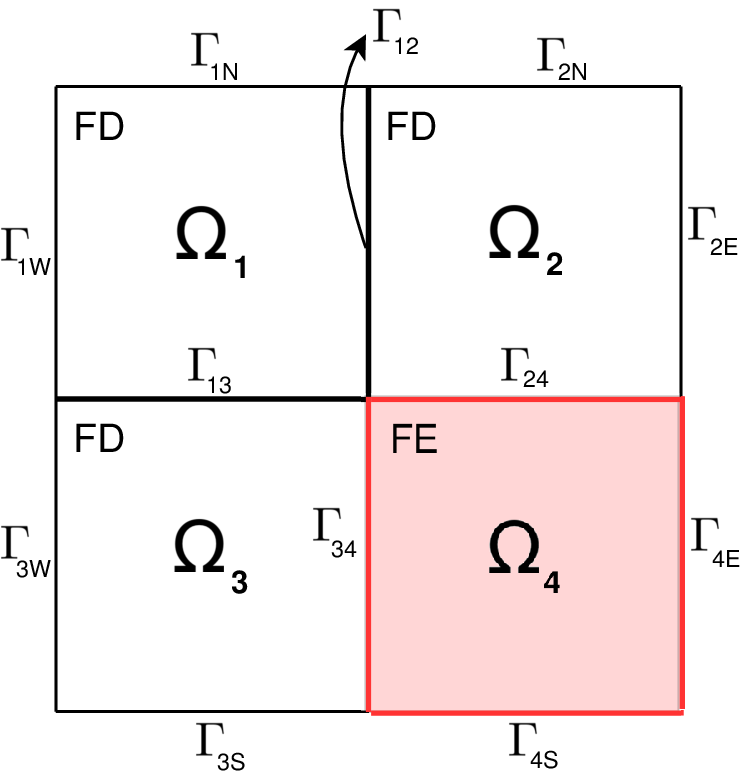}
    \caption{}
    \end{subfigure}
    \begin{subfigure}{0.32\textwidth}
    \centering
    \vspace{0.1cm}
    \includegraphics[scale=0.505]{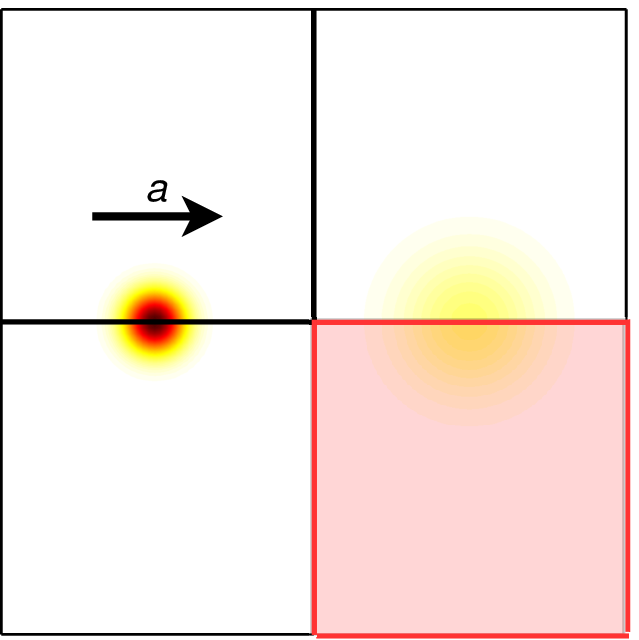}
    \caption{}
    \end{subfigure}
    \begin{subfigure}{0.32\textwidth}
    \centering
    \vspace{0.1cm}
    \includegraphics[scale=0.505]{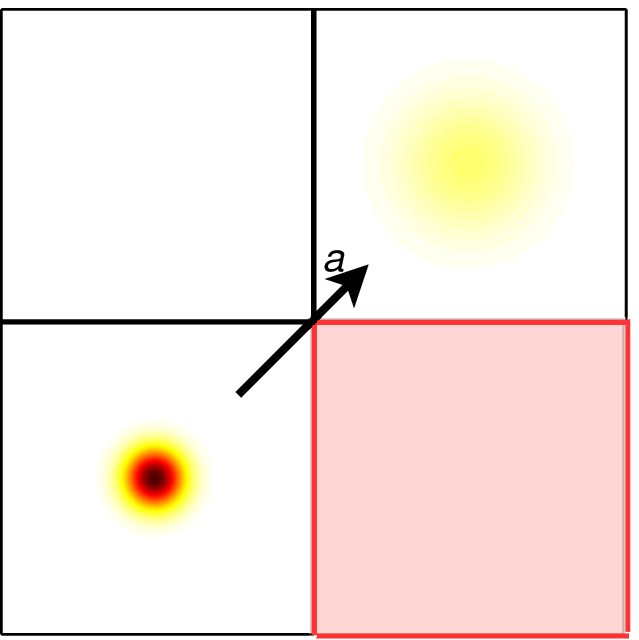}
    \caption{}
    \end{subfigure}
    \caption{The computational domain (a) and the two test cases (b) and (c).}
    \label{fig:2by2_domain}
\end{figure}

\begin{table}[pos=h]
\centering
\caption{Test case (b), Delaunay triangulated mesh, Fourth-order FD, $\mathbb{P}^1$ FE, $\bm a=(1,0)^T,x_0=-1,y_0=0,t=2$.}
\label{table:2by2_multiblock_test1}
\begin{tabular}{c|l|*4c|*4c}
\hline
\multirow{2}{*}{\scalebox{.8}{$m$ (FD)}} & \multirow{2}{*}{\scalebox{.8}{$r_{\max}$ (FE)}} & \multicolumn{4}{c|}{inviscid $\epsilon=0$} & \multicolumn{4}{c}{viscous $\epsilon=0.01$} \\ \cline{3-10}
& & \scalebox{.8}{$l_2e(\bm u_{FD})$} & \scalebox{.8}{$l_2e(\bm u_{FE})$} & \scalebox{.8}{$l_2e(\bm u)$} & $Q$ & \scalebox{.8}{$l_2e(\bm u_{FD})$} & \scalebox{.8}{$l_2e(\bm u_{FE})$}  & \scalebox{.8}{$l_2e(\bm u)$} & $Q$\\ \hline
41 & D/41 & 2.46E-2 & 5.25E-3 & 3.00E-2 & -- &6.45E-3 & 4.07E-3 & 1.06E-2 & -- \\
51 & D/51 & 1.27E-2 & 2.69E-3 & 1.54E-2 & 2.97 & 2.45E-3 &  1.25E-3 & 3.73E-3 & 4.66 \\
61 & D/61 & 7.92E-3 & 1.76E-3 & 9.68E-3 & 2.57 & 9.55E-4 & 3.16E-4 & 1.28E-3 & 5.86 \\
71 & D/71 & 5.49E-3 & 1.31E-3 & 6.79E-3 & 2.30 & 4.16E-4 & 4.36E-4 & 8.50E-4 & 2.66 \\
81 & D/81 & 3.98E-3 & 1.04E-3 & 5.00E-3 & 2.28 & 2.45E-4 & 5.00E-4 & 7.39E-4 & 1.05 \\
91 & D/91 & 2.88E-3 & 7.94E-4 & 3.71E-3 & 2.54 & 1.54E-4 & 4.07E-4 & 5.59E-4 & 2.37 \\ \hline
\end{tabular}
\end{table}

\begin{table}[pos=h]
\centering
\caption{Test case (c), Delaunay triangulated mesh, Fourth-order FD, $\mathbb{P}^1$ FE, $\bm a=(1,1)^T,x_0=-1,y_0=-1,t=2$.}
\label{table:2by2_multiblock_test2}
\begin{tabular}{c|l|*4c|*4c}
\hline
\multirow{2}{*}{\scalebox{.8}{$m$ (FD)}} & \multirow{2}{*}{\scalebox{.8}{$r_{\max}$ (FE)}} & \multicolumn{4}{c|}{inviscid $\epsilon=0$} & \multicolumn{4}{c}{viscous $\epsilon=0.01$} \\ \cline{3-10}
& & \scalebox{.8}{$l_2e(\bm u_{FD})$} & \scalebox{.8}{$l_2e(\bm u_{FE})$} & \scalebox{.8}{$l_2e(\bm u)$} & $Q$ & \scalebox{.8}{$l_2e(\bm u_{FD})$} & \scalebox{.8}{$l_2e(\bm u_{FE})$}  & \scalebox{.8}{$l_2e(\bm u)$} & $Q$\\ \hline
41 & D/41 & 4.26E-2 & 4.89E-3 & 4.77E-2 & -- & 1.73E-3 & 3.38E-7 & 1.74E-3 & -- \\
51 & D/51 & 2.29E-2 & 2.81E-3 & 2.59E-2 & 2.73 & 7.94E-4 & 1.25E-7 & 7.92E-4 & 3.54  \\
61 & D/61 & 1.38E-2 & 1.23E-3 & 1.50E-2 & 3.00 & 4.26E-4 & 1.20E-7 & 4.24E-4 & 3.43  \\
71 & D/71 & 8.91E-3 & 9.65E-4 & 9.89E-3 & 2.71 & 2.69E-4 & 1.14E-7 & 2.70E-4 & 2.92 \\
81 & D/81 & 6.30E-3 & 5.37E-4 & 6.78E-3 & 2.83 & 1.82E-4 & 1.00E-7 & 1.83E-4 & 2.91 \\
91 & D/91 & 4.78E-3 & 3.80E-4 & 5.11E-3 & 2.40 & 1.34E-4 & 6.02E-8 & 1.34E-4 & 2.68 \\ \hline
\end{tabular}
\end{table}

\begin{table}[pos=h]
\centering
\caption{Nonconformal Test case (b), Delaunay triangulated mesh, Fourth-order FD, $\mathbb{P}^1$ FE, $\bm a=(1,0)^T,x_0=-1,y_0=0,t=2$.}
\label{table:2by2_multiblock_test3}
\begin{tabular}{c|l|*4c|*4c}
\hline
\multirow{2}{*}{\scalebox{.8}{$m$ (FD)}} & \multirow{2}{*}{\scalebox{.8}{$r_{\max}$ (FE)}} & \multicolumn{4}{c|}{inviscid $\epsilon=0$} & \multicolumn{4}{c}{viscous $\epsilon=0.01$} \\ \cline{3-10}
& & \scalebox{.8}{$l_2e(\bm u_{FD})$} & \scalebox{.8}{$l_2e(\bm u_{FE})$} & \scalebox{.8}{$l_2e(\bm u)$} & $Q$ & \scalebox{.8}{$l_2e(\bm u_{FD})$} & \scalebox{.8}{$l_2e(\bm u_{FE})$}  & \scalebox{.8}{$l_2e(\bm u)$} & $Q$\\ \hline
21 & D/41 & 1.54E-1 & 3.39E-2 & 1.86E-1 & -- & 3.31E-2 & 1.55E-2 & 4.90E-2 & -- \\
26 & D/51 & 9.55E-2 & 2.00E-2 & 1.15E-1 & 2.20 & 2.45E-2 & 1.74E-2 & 4.17E-2 & 0.65 \\
31 & D/61 & 5.88E-2 & 1.17E-2 & 7.08E-2 & 2.67 & 1.48E-2 & 1.17E-2 & 2.69E-2 & 2.47 \\
36 & D/71 & 3.71E-2 & 7.41E-3 & 4.47E-2 & 3.00 & 9.33E-3 & 7.76E-3 & 1.70E-2 & 2.93 \\
41 & D/81 & 2.45E-2 & 5.01E-3 & 2.95E-2 & 3.09 & 5.62E-3 & 4.90E-3 & 1.05E-2 & 3.57 \\
46 & D/91 & 1.73E-2 & 3.39E-3 & 2.04E-2 & 3.05 & 3.31E-3 & 3.02E-3 & 6.31E-3 & 4.30 \\ \hline
\end{tabular}
\end{table}

\begin{table}[pos=h]
\centering
\caption{Nonconformal Test case (c), Delaunay triangulated mesh, Fourth-order FD, $\mathbb{P}^1$ FE, $\bm a=(1,1)^T,x_0=-1,y_0=-1,t=2$.}
\label{table:2by2_multiblock_test4}
\begin{tabular}{c|l|*4c|*4c}
\hline
\multirow{2}{*}{\scalebox{.8}{$m$ (FD)}} & \multirow{2}{*}{\scalebox{.8}{$r_{\max}$ (FE)}} & \multicolumn{4}{c|}{inviscid $\epsilon=0$} & \multicolumn{4}{c}{viscous $\epsilon=0.01$} \\ \cline{3-10}
& & \scalebox{.8}{$l_2e(\bm u_{FD})$} & \scalebox{.8}{$l_2e(\bm u_{FE})$} & \scalebox{.8}{$l_2e(\bm u)$} & $Q$ & \scalebox{.8}{$l_2e(\bm u_{FD})$} & \scalebox{.8}{$l_2e(\bm u_{FE})$}  & \scalebox{.8}{$l_2e(\bm u)$} & $Q$\\ \hline
21 & D/41 & 1.91E-1 & 5.25E-2 & 2.40E-1 & -- & 2.40E-2 & 1.58E-3 & 2.57E-2 & -- \\
26 & D/51 & 1.29E-1 & 3.09E-2 & 1.58E-1 & 1.88 & 9.33E-3 & 4.90E-5 & 9.33E-3 & 4.50 \\
31 & D/61 & 8.71E-2 & 1.62E-2 & 1.05E-1 & 2.37 & 4.68E-3 & 1.95E-6 & 4.68E-3 & 3.82 \\
36 & D/71 & 5.89E-2 & 1.00E-2 & 6.92E-2 & 2.69 & 2.69E-3 & 1.07E-6 & 2.69E-3 & 3.60 \\
41 & D/81 & 3.89E-2 & 6.76E-3 & 4.57E-2 & 3.03 & 1.66E-3 & 2.24E-7 & 1.66E-3 & 3.70 \\
46 & D/91 & 2.88E-2 & 3.98E-3 & 3.24E-2 & 2.83 & 1.05E-3 & 1.05E-7 & 1.05E-3 & 3.79 \\ \hline
\end{tabular}
\end{table}

\newpage

\newpage

\subsubsection{Superconvergence preservation}

For the convergence results, we use \textit{regular meshes} for the FE discretization, illustrated in Figure \ref{fig:computational_domain}, $\Omega = \Omega_L\cup\Omega_R= ([-2,0] \times [-1,1])\cup ([-2,0] \times [-1,1])$, where FD discretization is used on $\Omega_L$, and FE discretization is used on $\Omega_R$.

\begin{figure}[pos=h]
    \centering
    \includegraphics[scale=0.7]{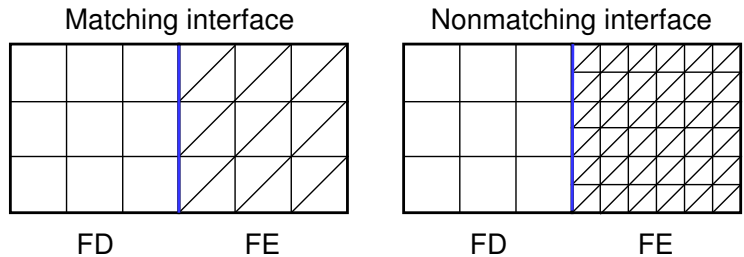}
    \caption{The computational mesh for superconvergence study.}
    \label{fig:computational_domain}
\end{figure}

\begin{table}[pos=h]
\caption{Superconvergence study on inviscid linear problem $\epsilon=0$, mesh type Figure \ref{fig:computational_domain} with $m$ grid points each direction each domain.}
\label{table:linear_superconvergence}
\centering
\begin{tabular}{c|*4c|*4c}
\hline
\multirow{2}{*}{$m$} & \multicolumn{4}{c|}{$\mathbb{P}^1$ FE to fourth-order FD} & \multicolumn{4}{c}{Fourth-order FD to $\mathbb{P}^1$ FE} \\ \cline{2-9}
& \scalebox{.8}{$l_2e(\bm u_{FD})$} & \scalebox{.8}{$l_2e(\bm u_{FE})$} & \scalebox{.8}{$l_2e(\bm u)$} & $Q$ & \scalebox{.8}{$l_2e(\bm u_{FD})$} & \scalebox{.8}{$l_2e(\bm u_{FE})$} & \scalebox{.8}{$l_2e(\bm u)$} & $Q$\\ \hline
31 & 1.86E-2 & 1.91E-2 & 3.72E-2 & -- & 1.95E-2 & 1.86E-2 & 3.80E-2 & --\\
41 & 7.59E-3 & 6.76E-3 & 1.45E-2 & 3.32 & 7.08E-3 & 7.59E-3 & 1.48E-2 & 3.28 \\
51 & 3.72E-3 & 2.69E-3 & 6.31E-3 & 3.67 & 2.88E-3 & 3.72E-3 & 6.61E-3 & 3.58 \\
61 & 2.14E-3 & 1.29E-3 & 3.39E-3 & 3.40 & 1.48E-3 & 2.14E-3 & 3.63E-3 & 3.26 \\ \hline
\end{tabular}
\end{table}

We test propagating the solution by both direction, from the FE domain to the FD domain $(x_0,y_0)^T=(1,0)^T,\bm a = (-1,0)^T$, and in the opposite direction $(x_0,y_0)^T=(-1,0)^T,\bm a = (1,0)^T$. \cgreen{The solution is evaluated at $t = 2$}. The convergence rates in fourth-order FD cases reported in Table \ref{table:linear_superconvergence} are noticeably close to a fourth-order--fourth-order coupling, cf., FD--FD fourth-order case in \citet{Mattsson2010}. We observe the so-called \textit{superconvergence} phenomenon in FE approximations, see e.g., \citet{Andreev1988}. At some special points, the order of convergence is higher than the convergence rate of the total error if calculated analytically. Those points are called \textit{superconvergence points}, and they happen to be the nodal points of the regular mesh Figure \ref{fig:computational_domain} with the advection operator $\bm a\cdot\grad v$. In the viscous case, there also exist \cblue{such} points which give rise to this behavior. It turns out that \eqref{eq:discrete_error_measurement} is not the most correct way to assess the discrete error in the sense of mimicking the continuous measurement $\norm{e}_{L^2(\Omega)} = \sqrt{\int_\Omega e^2 dxdy}$. A straightforward fix to avoid this problem is calculating the error with many more sample points added to the computation mesh. From a different angle, it cannot be denied that \eqref{eq:discrete_error_measurement} is desirable once our FE approximation fits in the SBP framework. (Simply, the defined discrete norm was applied.) To that extent, interpreting the advection operator as fourth-order and the diffusion operator as second-order may explain all the rates obtained. The authors have done separate numerical experiments on the FE approximation to verify that the current advection operator is fourth-order on the regular mesh regarding the error assessment \eqref{eq:discrete_error_measurement}. This is strong evidence that the proposed coupling technique preserves the native properties of the existing schemes.

\subsection{Nonlinear Burgers' Equation in 2D}
\label{sec:burger}
We consider the following nonlinear problem in two spatial dimensions
\begin{equation}
\begin{aligned}
\label{eq:burger_fem}
u_t + \bm a\cdot\grad\left(\frac{u^2}{2}\right) & = \grad\cdot(\epsilon\grad u), && (x,y) \in \Omega,\\
u & = u_0(x,y),&&(x,y) \in \Omega, t = 0.
\end{aligned}
\end{equation}
Inserting $\bm f = \left(\frac{u^2}{2}, \frac{u^2}{2}\right)^T$ to \eqref{eq:skew_symmetric_precise} gives $\alpha_1=\alpha_2=\frac{2}{3}$. From \eqref{eq:advection_diffusion_fem_bc}, a set of stable boundary condition reads
\begin{align}
\label{eq:burger_fem_bc}
\frac{1}{3}u(\bm a\cdot\mathbf{n}u-\abs{\bm a\cdot\mathbf{n}u})-\epsilon \mathbf{n}\cdot\grad u = g(x,y,t), \quad(x,y) \in \partial\Omega, t > 0,
\end{align}
where $\bm a = (a_1, a_2)^T>0$ is a constant vector and $\epsilon$ is a real constant.

\subsubsection*{The FD--FD coupling}
A FD--FD semi-discretization of the coupling problem \eqref{eq:advection_diffusion_coupling_continuous} under the governing equation \eqref{eq:burger_fem} is given by
\begin{equation}
\label{eq:burger_sbp_sat}
\begin{aligned}
\bm u_t &= -\frac{1}{3}a_1{\bm U}\D_x\bm u -\frac{1}{3}a_1\D_x{\bm U}\bm u - \frac{1}{3}a_2{\bm U}\D_y\bm u -\frac{1}{3}a_2\D_y{\bm U}\bm u + \epsilon \D_{2x}\bm u + \epsilon \D_{2y}\bm u + \mbox{SAT}_L+\mbox{SAT}_{IL},\\
\bm v_t &= -\frac{1}{3}a_1\bm V\D_x\bm v -\frac{1}{3}a_1\D_x\bm V\bm v - \frac{1}{3}a_2\bm V\D_y\bm v -\frac{1}{3}a_2\D_y\bm V\bm v + \epsilon \D_{2x}\bm v + \epsilon \D_{2y}\bm v + \mbox{SAT}_R+\mbox{SAT}_{IR},
\end{aligned}
\end{equation}
where the interface SATs are
\begin{equation}
\label{eq:burger_sat_interface}
\begin{aligned}
\mbox{SAT}_{IL} & = \frac{1}{6} ({\bm H_x^L}^{-1} \bm e_n)\otimes (({\bm U}\bm u)_I-\bm{I}_{R2L}(\bm V\bm v)_I)+\frac{1}{6} ({\bm H_x^L}^{-1} \bm e_n)\otimes \left[{\bm U_I}(\bm u_I-\bm{I}_{R2L}\bm v_I)\right]\\
&+ \delta_L ({\bm H_x^L}^{-1} \bm e_n) \otimes ((\D_xu)_I - \bm{I}_{R2L}(\D_x\bm v)_I) + \sigma_L ({\bm H_x^L}^{-1}\bm d_n^T)\otimes(\bm u_I-\bm{I}_{R2L}\bm v_I),\\
\mbox{SAT}_{IR} & = -\frac{1}{6}({\bm H_x^R}^{-1} \bm e_1) \otimes ((\bm V\bm v)_I-\bm{I}_{L2R}({\bm U}\bm u)_I)-\frac{1}{6}({\bm H_x^R}^{-1} \bm e_1) \otimes \left[\overline{\bm v_I}(\bm v_I-\bm{I}_{L2R}\bm u_I)\right]\\
&+\delta_R ({\bm H_x^R}^{-1} \bm e_1) \otimes ((\D_x\bm v)_I-\bm{I}_{L2R}(\D_x\bm u)_I)+ \sigma_R ({\bm H_x^R}^{-1} \bm d_1^T) \otimes (\bm v_I-\bm{I}_{L2R}\bm u_I),
\end{aligned}
\end{equation}
where $\delta_L=-\sigma_R, \sigma_L =-\delta_R, \delta_L+\sigma_L = -1$.
\subsubsection*{The FD--FE coupling}
A FD--FE semi-discretization of the coupling problem \eqref{eq:advection_diffusion_coupling_continuous} under the governing equation \eqref{eq:burger_fem} is given by
\begin{equation}
\label{eq:burger_fe_sbp_sat}
\begin{aligned}
\bm u_t &= -\frac{1}{3}a_1{\bm U}\D_x\bm u -\frac{1}{3}a_1\D_x{\bm U}\bm u - \frac{1}{3}a_2{\bm U}\D_y\bm u -\frac{1}{3}a_2\D_y{\bm U}\bm u + \epsilon \D_{2x}\bm u + \epsilon \D_{2y}\bm u + \mbox{SAT}_L+\mbox{SAT}_{IL},\\
\bm v_t &= -\mathbf{M}^{-1}\mathbf{C}(\bm v)\bm v-\epsilon \mathbf{M}^{-1}\mathbf{A}\bm v+\frac{a_2}{3}\mathbf{M}^{-1}[\mathscr{R}_M^{N}(\bm v)-\mathscr{R}_M^{N}(\abs{\bm v})]\bm v\\
&\hspace{0.5cm}+\frac{a_1}{3}\mathbf{M}^{-1}[\mathscr{R}_M^{E}(\bm v)-\mathscr{R}_M^{E}(\abs{\bm v})]\bm v+\frac{a_2}{3}\mathbf{M}^{-1}[-\mathscr{R}_M^{S}(\bm v)-\mathscr{R}_M^{S}(\abs{\bm v})]\bm v - \bm r +\mbox{SAT}_{IR},
\end{aligned}
\end{equation}
where $\mathscr{R}_M^N,\mathscr{R}_M^E,\mathscr{R}_M^S$ are obtained by limiting the integration of $\mathscr{R}_M$ to $\Gamma_N,\Gamma_E,\Gamma_S$, respectively. The coupling penalty terms are given by
\begin{equation}
\label{eq:burger_fe_interface}
\begin{aligned}
\mbox{SAT}_{IL} & = \frac{1}{6} (H_x^{-1} \bm e_n)\otimes (({\bm U}\bm u)_I-\bm{I}_{R2L}\bm L_W\bm V\bm v)+\frac{1}{6} (H_x^{-1} \bm e_n)\otimes \left[{\bm U_I}(\bm u_I-\bm{I}_{R2L}\bm L_W\bm v)\right]\\
&+ \delta_L\epsilon(H_x^{-1} \bm e_n) \otimes ((\D_x\bm u)_I - \bm{I}_{R2L}\mathbf{D}_x^W\bm v)+\sigma_L\epsilon(H_x^{-1}\bm d_n^T)\otimes(\bm u_I-\bm{I}_{R2L}\bm L_W\bm v),\\
\mbox{SAT}_{IR} & = -\frac{1}{6}\mathbf{M}^{-1}\mathbf{R}_M^W(\bm V\bm v-\bm L_W^T\bm{I}_{L2R}({\bm U}\bm u)_I)-\frac{1}{6}\mathbf{M}^{-1}\bm V\mathbf{R}_M^W\bm L_W^T(\bm L_W\bm v-\bm{I}_{L2R}\bm u_I)\\
&+\delta_R\epsilon\mathbf{M}^{-1}\mathbf{R}_M^W\bm L_W^T(\mathbf{D}_x^W\bm v-\bm{I}_{L2R}(\D_x\bm u)_I) +\sigma_R\epsilon\mathbf{M}^{-1}(-\mathbf{R}^W_C)^T\bm L_W^T(\bm L_W\bm v - \bm{I}_{L2R}\bm u_I),
\end{aligned}
\end{equation}
where $\delta_L=-\sigma_R, \sigma_L =-\delta_R, \delta_L+\sigma_L = -1$.
\subsubsection*{Convergence study}
An analytic solution of \eqref{eq:burger_fem} subject to the boundary condition \eqref{eq:burger_fem_bc} is
\[
\cblue{u(x,y,t) = \frac{c}{a_1^2+a_2^2} - b \tanh\left(b\frac{a_1(x-x_0)+a_2(y-y_0)-ct}{2\epsilon}\right), \quad t > 0,}
\]
where $b$ and $c$ are arbitrary constants. For the following convergence study, we choose $b = 1, c = 2$. The semi-discrete scheme \eqref{eq:burger_fe_sbp_sat} is integrated in time using the standard Runge-Kutta 4. The time step size is chosen to be $\Delta t = \beta h^2, \beta = 0.1$. The domain setup is as in Figure \ref{fig:coupling-interface}, where $\Omega_L=[-2,0]\times[-1,1]$, and $\Omega_R=[0,2]\times[-1,1]$. The problem settings are: $\cblue{\bm a = (1,1)^T}, \epsilon = 0.1,x_0=-1,\cblue{y_0=-1}$. The solution is examined at $t = 1$. A convergence study is shown in Tables \ref{table:nonlinear_convergence_fd_fe} and \ref{table:nonlinear_convergence_fd_fe_nonmatching}. The notations in Tables \ref{table:nonlinear_convergence_fd_fe} and \ref{table:nonlinear_convergence_fd_fe_nonmatching} are same to the ones used in Table \ref{table:2by2_multiblock_test1}. In both tables, we can observe the expected convergence rate of second-order. It can also be seen that the use of higher-order FD operators slightly lowers the error and improves convergence rates.

\begin{table}[pos=h]
\centering
\caption{Convergence rates for the case of matching interface.}
\label{table:nonlinear_convergence_fd_fe}
\cblue{
\begin{tabular}{c|*4c|*4c}
\hline
\multirow{2}{*}{$m$ (FD,FE)} & \multicolumn{4}{c|}{Second-order FD} & \multicolumn{4}{c}{Fourth-order FD} \\ \cline{2-9}
& \scalebox{.8}{$l_2e(\bm u_{FD})$} & \scalebox{.8}{$l_2e(\bm u_{FE})$} & \scalebox{.8}{$l_2e(\bm u)$} & $Q$ & \scalebox{.8}{$l_2e(\bm u_{FD})$} & \scalebox{.8}{$l_2e(\bm u_{FE})$} & \scalebox{.8}{$l_2e(\bm u)$} & $Q$\\ \hline
21 & 2.91E-02 & 1.83E-02 & 4.73E-02 & -- & 1.80E-02 & 1.73E-02 & 3.53E-02 & -- \\
31 & 1.31E-02 & 9.19E-03 & 2.23E-02 & 1.85 & 7.64E-03 & 8.81E-03 & 1.64E-02 & 1.89 \\
41 & 7.43E-03 & 5.50E-03 & 1.29E-02 & 1.90 & 4.22E-03 & 5.30E-03 & 9.52E-03 & 1.89 \\
51 & 4.77E-03 & 3.64E-03 & 8.42E-03 & 1.91 & 2.67E-03 & 3.52E-03 & 6.19E-03 & 1.93 \\ \hline
\end{tabular}}
\end{table}

\begin{table}[pos=h]
\centering
\caption{Convergence rates for the case of 2 to 1 mesh ratio nonconformal coupling.}
\label{table:nonlinear_convergence_fd_fe_nonmatching}
\cblue{
\begin{tabular}{c|c|*4c|*4c}
\hline
\multirow{2}{*}{$m$ (FD)} & \multirow{2}{*}{$r_{\max}$ (FE)} & \multicolumn{4}{c|}{Second-order FD} & \multicolumn{4}{c}{Fourth-order FD} \\ \cline{3-10}
& & \scalebox{.8}{$l_2e(\bm u_{FD})$} & \scalebox{.8}{$l_2e(\bm u_{FE})$} & \scalebox{.8}{$l_2e(\bm u)$} & $Q$ & \scalebox{.8}{$l_2e(\bm u_{FD})$} & \scalebox{.8}{$l_2e(\bm u_{FE})$} & \scalebox{.8}{$l_2e(\bm u)$} & $Q$\\ \hline
11 & D/21 & 1.10E-01 & 2.14E-02 & 1.31E-01 & -- & 6.44E-02 & 2.11E-02 & 8.54E-02 & -- \\
21 & D/41 & 2.49E-02 & 5.93E-03 & 3.09E-02 & 2.08 & 8.44E-03 & 5.59E-03 & 1.40E-02 & 2.61 \\
31 & D/61 & 1.07E-02 & 2.75E-03 & 1.34E-02 & 2.06 & 3.11E-03 & 2.61E-03 & 5.72E-03 & 2.21 \\ \hline
\end{tabular}}
\end{table}

\section{Conclusion}
\label{sec:conclusion}
\cpurple{The multiblock coupling of FD schemes and FE schemes in this paper is presented within the SBP framework. The technique uses a SAT technique to impose the interface continuity weakly. SBP-preserving interpolation operators resolve the nonconforming interface distribution.} Accuracy, stability, and conservation are investigated on general scalar conservation laws. Numerical results are shown for a linear hyperbolic/parabolic problem and a nonlinear problem.\\

Future works may concern extensions for curvilinear grids, optimal time-stepping, an option for upwind operators \cite{Lundgren2020}, high-order finite elements or a solution to the case when block corners are unmatching. The work can be structurally extended for nonlinear systems.

\appendix
\section{Appendix: Construction of the nondiagonal norm interpolation operators}
\label{section:construct_interpolation_operators}
We show two ways of constructing the necessary interpolation operators in the interface SATs in this section. The first approach is similar to the one introduced in \citet{Mattsson2010}. The second way requires constructing a pair of diagonal norm interpolation matrices.
\subsection{Construction from the matrix operator structure}
\label{section:construction_by_structure}
We briefly show the construction of interpolation operators satisfying \eqref{eq:conservation_property_fd_fe} and \eqref{eq:accuracy_requirement}. If the solutions are coupled along characteristic lines, see \citet{Mattsson2010}, an extra requirement reads
\begin{equation}
\label{eq:sylvester}
\bm H (\bm I-\bm I_{\textit{FD2FE}}\bm I_{\textit{FE2FD}}) \geq 0,\;\; \mathbf M_I (\bm I-\bm I_{\textit{FE2FD}}I_{\textit{FD2FE}}) \geq 0,
\end{equation}
where \cblue{$\mathbf M_I$ is defined in Definition \ref{def:conservation_property_fd_fe}, $\bm H$ is the one-dimensional FD norm matrix along the interface}, $\bm I_{\textit{FD2FE}}$ translates the nodal values from a FD interface to a FE interface, and $\bm I_{\textit{FE2FD}}$ does the opposite direction.

The structure of $\bm I_{\textit{FE2FD}}$ is given by
\setcounter{MaxMatrixCols}{20}
\[
\bm I_{\textit{FE2FD}} = \begin{bmatrix}
q_{1,1} & \dots & & & & & & \dots & q_{1,r} & & & & & & \\
\vdots & & & & & & & & \vdots & & & & & & \\
q_{t,1} & \dots & & & & & & \dots & q_{t,r} & & & & & & \\
& & q_{2s} & q_{2s-1} & \dots & q_0 & \dots & q_{2s-1} & q_{2s} & & & & & & \\
& & & q_{2s} & q_{2s-1} & \dots & q_0 & \dots & q_{2s-1} & q_{2s} & & & & & \\
& & & & \ddots & \ddots & \ddots & \ddots & \ddots & \ddots & & & & & \\
& & & & & q_{2s} & q_{2s-1} & \dots & q_0 & \dots & q_{2s-1} & q_{2s} & & & \\
& & & & & & q_{2s} & q_{2s-1} & \dots & q_0 & \dots & q_{2s-1} & q_{2s} & & \\
& & & & & & q_{t,r} & \dots & & & & & & \dots & q_{t,1} \\
& & & & & & \vdots & & & & & & & & \vdots \\
& & & & & & q_{1,r} & \dots & & & & & & \dots & q_{1,1} \\
\end{bmatrix},
\]
where $r,m,s$ depend on interface nodal distribution and accuracy orders. For example, when using $6^{th}$ order FD and $\mathbb{P}^1$ FE, two to one interface ratio (finer on the FE side), we have used $r=18$, $t=6$, $s=7$. The other corresponding operator is given by $\bm I_{\textit{FD2FE}} = \bm H^{-1} (\mathbf M_I\bm I_{\textit{FE2FD}})^T$. Given the known structure, $\bm I_{\textit{FE2FD}}$ and $\bm I_{\textit{FD2FE}}$ can be determined by solving for accuracy \eqref{eq:accuracy_requirement} using some mathematical symbolic package (such as Maple or Symbolic Matlab).

\subsection{Construction from existing diagonal norm interpolation operators}\label{section:shortcut_diagonal_norm}
A simpler construction can be derived ultilizing the following property. Given that $\bm I_{b2c}, \bm I_{a2b}$ are SBP-preserving interpolation operators, assuming compatible matrix multiplication, we have that
\begin{equation}
\label{eq:3_layers}
\bm I_{a2c}=\bm I_{b2c}\bm I_{a2b}
\end{equation}
is also an SBP-preserving interpolation operator,  see \citet{Kozdon2016}. Our idea is to fabricate a one-dimensional middle layer of which nodal distribution matches the FE interface nodes, and the corresponding one-dimensional norm in the $y$-direction reassembles a FD diagonal norm. Conveniently, the difference scheme corresponding to the middle layer needs not be precisely constructed. An illustration is shown in Figure \ref{fig:nondiagonal_interpolation}. A simplest example of this idea is shown in Example \ref{example:2nd_fd_p1_fe_int_op}. \cblue{The authors found that the technique sketched here is related to the mortar method, see \cite{Kopriva1996, Bui-Thanh2012}, the mortar method in SBP context \cite{Chan2020,Friedrich2018}.}
\begin{figure}[pos=h]
    \centering
    \includegraphics[scale=2]{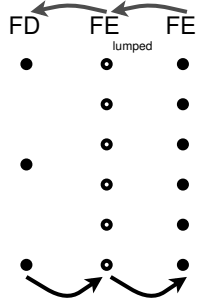}
    \caption{Construction of nondiagonal norm interpolation operators using a "bridge" diagonal norm layer.}
    \label{fig:nondiagonal_interpolation}
\end{figure}

\begin{example}
\label{example:2nd_fd_p1_fe_int_op}
\textbf{$2^{nd}$ order FD - $\mathbb{P}^1$ FE matching uniformly distributed interface}

Recall the notations $\mathbf{M}_I,\bm H$ being the one-directional norms in the $y$-direction of the FE scheme and the FD scheme, respectively, $h_y$ being the step size, "L", "R", "M" being the FD, FE layers, and the lumped FE layer. We have
\[
\mathbf{M}_I = h_y\begin{bmatrix}
1/3 & 1/6 & 0 & 0 & \hdots & 0 \\
1/6 & 2/3 & 1/6 & 0 & \hdots & 0 \\
0 & 1/6 & 2/3 & 1/6 & \hdots & 0 \\
\vdots & 0 & \ddots & \ddots & \ddots & \vdots\\
0 & \hdots & 0 & 1/6 & 2/3 & 1/6\\
0 & \hdots & 0 & 0 & 1/6 & 1/3
\end{bmatrix},\quad
\bm H = h_y\begin{bmatrix}
0.5 & 0 & 0 & 0 & \hdots & 0 \\
0 & 1 & 0 & 0 & \hdots & 0 \\
0 & 0 & 1 & 0 & \hdots & 0 \\
\vdots & 0 & \ddots & \ddots & \ddots & \vdots\\
0 & \hdots & 0 & 0 & 1 & 0\\
0 & \hdots & 0 & 0 & 0 & 0.5
\end{bmatrix}.
\]

For the interpolation between FE$_{\text{lumped}}$ and FE, we use $\bm{I}_{M2R}$ being the identity matrix and
\[
\bm{I}_{R2M} = \begin{bmatrix}
2/3 & 1/3 & 0 & 0 & \hdots & 0 \\
1/6 & 2/3 & 1/6 & 0 & \hdots & 0 \\
0 & 1/6 & 2/3 & 1/6 & \hdots & 0 \\
\vdots & 0 & \ddots & \ddots & \ddots & \vdots\\
0 & \hdots & 0 & 1/6 & 2/3 & 1/6\\
0 & \hdots & 0 & 0 & 1/3 & 2/3
\end{bmatrix}.
\]
\end{example}

\begin{theorem}
\label{theorem:construction_general_procedure}
Assume that the diagonal norm interpolation operators $\bm{I}_{L2M},\bm{I}_{M2L}$ fulfill the acuracy and SBP requirements \eqref{eq:accuracy_requirement}, \eqref{eq:conservation_property_fd_fd}, \eqref{eq:sylvester}. By constructing
\[
\bm{I}_{L2R}=\bm{I}_{M2R}\bm{I}_{L2M} \text{ and } \bm{I}_{R2L}=\bm{I}_{M2L}\bm{I}_{R2M},
\]
where $\bm{I}_{M2R} = \bm I_{m\times n}, \bm{I}_{R2M} = (\mathbf M_I \bm H^{-1})^T$, by Definition \ref{def:conservation_property_fd_fd}, $\bm{I}_{L2R}, \bm{I}_{R2L}$ are SBP-preserving interpolation operators. Moreover, they also fulfill the extra requirement \eqref{eq:sylvester}.
\end{theorem}

Since the lumped $\mathbb{P}^1$ mass matrix corresponds to a second-order SBP FD diagonal norm, and the corresponding interpolation operators are of the form in Example \ref{example:2nd_fd_p1_fe_int_op} reassembling the Simpson's rule naturally hold the properties \eqref{eq:conservation_property_fd_fd}, \eqref{eq:accuracy_requirement} to the desired accuracy, and \eqref{eq:sylvester}. By the relation \eqref{eq:3_layers}, Theorem \ref{theorem:construction_general_procedure} always holds true.

\section*{Acknowledgments}
The authors acknowledge the on-point comments from the anonymous reviewers which have helped improve the quality of this paper. The authors would like to thank Vidar Stiernstr{\"o}m and Lukas Lundgren for the fruitful discussions.


\bibliographystyle{cas-model2-names}

\bibliography{references}

\end{document}